\documentclass[12pt]{amsart}
\usepackage{amsfonts,latexsym,rawfonts,amsmath,amssymb,amsthm}
\usepackage[plainpages=false]{hyperref}
\usepackage{mathrsfs}
\usepackage{enumerate}

\usepackage{graphicx}

\RequirePackage{color}

\newcommand\Om{\Omega}
\newcommand\pom{\partial\Omega}
\newcommand\bom{\overline\Omega}
\newcommand\R{\Bbb R}
\newcommand\eps{\varepsilon}
\newcommand\ve{\varepsilon}

\newcommand\p{\partial}

\newcommand\lan{\langle}
\newcommand\ran{\rangle}

\newcommand\beq{\begin{equation}}
\newcommand\eeq{\end{equation}}

\newtheoremstyle{mythm}{1.5ex plus 1ex minus .2ex}{1.5ex plus 1ex minus .2ex}{\kai}{\parindent}{\song\bfseries}{}{1em}{}

\numberwithin{equation}{section}

\newtheorem{theorem}{Theorem}[section]
\newtheorem{lemma}{Lemma}[section]
\newtheorem{example}{Example}[section]
\newtheorem{proposition}{Proposition}[section]
\newtheorem{remark}{Remark} [section]
\newtheorem{corollary}{Corollary}[section]
\allowdisplaybreaks[4]

\begin{document}
\title[Free boundary for the Monge-Amp\`ere equation]
{Regularity of free boundary \\ for the Monge-Amp\`ere obstacle problem}
%%%% \date\today

\author[G. Huang]
{Genggeng Huang}
\address
{School of Mathematical Sciences,
Fudan University,
Shanghai 200433, China.}
\email{genggenghuang@fudan.edu.cn}

\author[L. Tang]
{Lan Tang}
\address
	{School of Mathematics and Statistics, Central China Normal University, Wuhan 430079, China.}
\email{lantang@mail.ccnu.edu.cn}

\author[X.-J. Wang]
{Xu-Jia Wang}
\address
{Centre for Mathematics and Its Applications,
The Australian National University,
Canberra, ACT 0200, Australia}
\email{Xu-Jia.Wang@anu.edu.au}

\thanks{The first author was supported by NNSFC 11871160, 
the second author was supported by NNSFC 11831009 and NNSFC 12171185,
the third author was supported by ARC DP200101084.}

\subjclass[2000]{35J96, 35R35, 35B65.}

\keywords{Monge-Amp\`ere equation, free boundary, regularity.}

\maketitle

\begin{abstract}
In this paper, we prove the regularity of the free boundary in the Monge-Amp\`ere obstacle problem 
$\det D^2 v= f(y)\chi_{\{v>0\}}. $
By duality, the regularity of the free boundary is equivalent to that of the asymptotic cone of the solution to 
the singular Monge-Amp\`ere equation
$\det D^2 u = 1/f (Du)+\delta_0$
at the origin.
We first establish an asymptotic estimate for the solution $u$ near the singular point,
then use a partial Legendre transform to change the Monge-Amp\`ere equation to
a singular,  fully nonlinear elliptic  equation, 
and establish the regularity of solutions to the singular elliptic equation. 
\end{abstract}

\baselineskip16.8pt
\parskip5pt

\section{\bf Introduction}
%Section1

In this paper we study  the Monge-Amp\`ere obstacle problem
\begin{equation}\label{MAob}
{\begin{split}
\det D^2 v & = f \, \chi_{\{v>0\}}\ \ \text{in}\ \Omega,\\
 v & = {v_0} \ \ \hskip31pt \text{on}\ \pom ,
 \end{split}}
\end{equation}
%%%  v &\ge 0 \ \ \hskip36pt  \text{in}\ \Om, 
where $\Omega$ is a bounded domain in the Euclidean space $\Bbb R^n$,
$f, {v_0}$ are positive functions on $\bom$, and $\chi$ is the characteristic function.
Denote $\Gamma :=\p\{v=0\}$ the free boundary. 
The main objective of the paper is to prove the regularity of the free boundary $\Gamma$.

Problem \eqref{MAob} is the Monge-Amp\`ere counterpart of the classical free boundary problem
\begin{equation}\label{Lapob}
{\begin{split}
\Delta v & = f \, \chi_{\{v>0\}}\ \ \text{in}\ \Omega,\\
%%%  v &\ge 0 \ \ \hskip35pt  \text{in}\ \Om,\\
 v & = {v_0} \ \ \hskip30pt \text{on}\ \pom . 
 \end{split}}
\end{equation}
A central issue for the prototypal obstacle problem \eqref{Lapob} is the regularity of the free boundary $\Gamma=\p\{v=0\}$.
In a seminal work \cite{Ca77}, Caffarelli proved that the free boundary is $C^1$ smooth at regular points; 
and hence is $C^\infty$ smooth and analytic  \cite{KN77}.
Since then the regularity of free boundary problems has been extensively studied.

%+% In \cite {CKS}, the authors proved the regularity of free boundary  
%+% for  a problem resembling the obstacle problem \eqref{Lapob} but with  no sign assumption.
%+% In \cite {CR, FRS, FSe, W99}, 
%+% the authors studied the geometric behaviour of the free boundary of \eqref{Lapob} at singular points.
%+% The regularity of free boundaries arising in different applications has also been studied by many people,
%+% such as the fractional Laplacian \cite{CSS},  an integro-differential equation \cite {CRS},  
%+% the porous medium equation \cite{DH1}, and  optimal transportation \cite {CM}.

%%% We refer the readers to two recent deep works \cite {FRS, FSe} for more details on this research thread.
%%%  Similarly to   \eqref{cob1}, problem \eqref{MAob} can also be reformulated as a variational problem
%%% $$\inf\Big\{\int_\Om \big[(-w)\det D^2 w + (n+1)fw\big]:\  w\in {\mathcal W'}\Big\}$$
%%% where 
%%% ${\mathcal W'}=\{w\in C(\bom)\ \text{is convex}, \ w={v_0}\  \text{on}\ \p\Omega\ \text{and}\ w\ge 0\ \text{in}\ \Omega\}$.
%%% Moreover, the solution is given by 
%%% \begin{equation*}v=\inf_{w\in\mathcal W''} w,\quad
%%%  \mathcal W'' =\{w\in \mathcal W', \ \mathcal M w\le f\}
%%%  \end{equation*}
%%% and $\mathcal M$ denotes the Monge-Amp\`ere operator \cite{S05}.
%%% {\color{red} (Is it quite easy to see that the minimizer of the variational problem implies the obstacle problem?)}

Regularity of  free boundary problems associated with the Monge-Amp\`ere equation 
has also been studied in a number of papers \cite{CM, ChW, GJM15, GJM16, Lee, LLP, S05}. 
See also \cite{DH2, DL1, Ham} for  a related free boundary problem with the Gauss curvature flow. 
In \cite{S05}, Savin studied problem \eqref{MAob} 
and proved that the free boundary $\Gamma$ is uniformly convex and $C^{1,1}$ smooth.
He also pointed out two other interpretations of the obstacle problem \eqref{MAob},
as a model in the optimal transportation with a Dirac measure 
and in the Monge-Amp\`ere equation with a cone singularity.
In dimensions two, Galvez, Jim\'enez  and Mira  \cite{GJM15} proved that the free boundary $\Gamma$
is $C^\infty$ smooth and analytic, if $f$ is respectively smooth and analytic.
As pointed out  in \cite{GJM15},
some of the arguments  in \cite{GJM15} are specific of the two-dimensional case, 
since they rely on complex analysis and surface theory. 
%%% In \cite{CM, CLW, Lee}, different free boundary problems for the Monge-Amp\`ere equation were studied 
%%% and the $C^{1,\alpha}$ regularity of the free boundary was obtained. 
Also in dimension two, Daskalopoulos and Lee \cite{DL1} obtained the regularity of the free boundary 
in the Gauss curvature flow. 
For fully nonlinear, uniformly elliptic equations, 
the corresponding results can be found in \cite{Lee98, SY1, SY2}.
An open problem is the regularity of the free boundary problem \eqref{MAob} in high dimensions.
In this paper we resolve the problem completely.

\begin{theorem}\label{thmA}
Let $v$ be a generalized solution to the obstacle problem \eqref{MAob}, in the sense of Aleksandrov. 
Assume that $\Omega$ is a bounded %convex 
domain in $\Bbb R^n$, and $f, v_0>0$.
Then the free boundary $\Gamma$ is smooth if $f$ is smooth; and $\Gamma$ is analytic if $f$ is analytic.
\end{theorem}
 
Monge-Amp\`ere type equations are different from uniformly elliptic equations in many ways,
and the techniques for uniformly elliptic equations  \cite{Ca77, Ca98, Lee98, SY1, SY2} 
do not apply to the Monge-Amp\`ere obstacle problem \eqref{MAob}.
A key property for uniformly elliptic equations is the growth estimate $\sup_{B_r(0)} v\approx r^2$ 
(assuming that $0$ is a point on the free bounadry).
By this property, the blow-up profile at $0$ is either of the form 
$\frac12 \max\{x\cdot e, 0\}^2$ for a vector $e$, or of the form $\frac12 x\cdot Ax$ for a matrix $A$. 
In the former case, $0$ is a regular point and the free boundary is smooth at the point.
But for the Monge-Amp\`ere equation  \eqref{MAob}, we have instead the estimate $\sup_{B_r(0)} v\approx r^{1+\frac{1}{n}}$.
It implies that equation  \eqref{MAob} is both singular and degenerate near the free boundary.
Therefore to prove the regularity of the free boundary, one needs a completely different approach.

Our new approach to the problem involves a series of transforms.
We first make the Legendre transform \eqref{Leg-Ch} to obtain the equation  \eqref{MAs0} with point singularity.
By estimate \eqref{wx} we then make the change \eqref{Trans2} to get equation \eqref{MAs1} in the polar coordinates.
Regularity of the free boundary $\Gamma$ is thus reduced to that for equation \eqref{MAs1}.
To establish the a priori estimate for \eqref{MAs1},
we again make a partial Legendre transform \eqref{Par-Leg} to obtain a fully nonlinear, singular elliptic equation \eqref{MAs3}.
The linearized equation of \eqref{MAs3} is a degenerate equation of  Keldysh type 
after a change of variables.
%%% Similarly to the Tricomi equation,  equations of Keldysh type is a basic mixed type equation.
%%% It is elliptic in the domain $\{x_n>0\}$ and hyperbolic in  $\{x_n<0\}$.
By the H\"older regularity of the linearized equation of \eqref{MAs3},
we can prove the blow-up limits for equation \eqref{MAs3} is a quadratic function (Theorem \ref{thmC}).
Technically the most difficult part in the approach is to prove the uniqueness of the blow-up limits, 
namely all blow-up sequences around a point converge to the same limit,
from which the $C^2$ regularity of solutions follows.
%%%Recall that for the classical obstacle problem \eqref{Lapob},
%%%the key step is to prove the $C^1$ regularity from Lipschitz continuity \cite{Ca77}.
We will prove  the uniqueness of the blow-up limits by consecutively using the maximum principle in an infinite sequence of domains.
See also Remark \ref{r4.4}.

The Monge-Amp\`ere equation has a very useful property, namely the duality.
Let $v$ be the solution to the obstacle problem \eqref{MAob}, as in Theorem \ref{thmA}.
Then $v$ is convex and the set $\{v=0\}$ is a convex sub-set of $\Om$.
We may assume that the set $\{v=0\}$ has positive measure, 
otherwise the free boundary does not exist \cite{S05}. 
Let $u$ be the Legendre transform of $v$, given by
\beq\label{Leg-Ch}
 u(x) =\sup\{x\cdot y - v(y):\ y\in \Om\} , \ \ \ x\in\Om^*=:D_yv(\Om).
 \eeq
Then $u$ is a generalized solution to the following Monge-Amp\`ere equation with point singularity,
\begin{equation}\label{MAs0}
\det D^2 u=g(Du)+c^*\delta_0\ \ \text{in}\ \Om^*,
\end{equation}
where $g(Du(x))=\frac{1}{f(y)}$ at $y=Du(x)$, and 
$c^*=|\{v=0\}|$ is a constant. There is no loss of generality in assuming that $c^*=1$.
By a translation of the coordinates we assume that
the origin is an interior point of the convex set $\{v=0\}$. 
Then $u(0)=0$ and $u(y)> 0\ \forall\ y\ne 0$. 
Let $\phi$ be the tangential  cone of $u$ at $0$, namely it is a
homogeneous function of degree one defined in $\R^n$ and satisfying
\beq\label{tc}
{\begin{split}
  & u(x)\ge \phi(x)\ \ \  \forall\ \ x\in \Om^*,\\
  & u(x)-\phi(x)=o(|x|)\ \text{as}\ \ x\to 0,
  \end{split}} 
\eeq
such that $\p\phi\{0\}=\p u\{0\} $.
Given a convex function $w$, we denote by $\p w\{p\}$ the sub-differential of $w$ at $p$, 
$$\p w\{p\}=\{\xi \in\R^n:\ w(x)\ge \xi\cdot (x-p) +w(p)\ \forall\ x\ \text{near}\ p \}. $$
By duality, 
$$\p\phi\{0\} =\{v=0\}  . $$
Hence the section $S_{1,\phi} =: \{x\in\R^n:\ \phi(x)<1\}$ is the polar body of $\{v=0\}$,  i.e., 
\begin{equation}\label{dual}
S_{1, \phi} =\{x\in\mathbb R^n:\ \ x\cdot y< 1\ \ \forall\ y\in \{v=0\}\}. 
\end{equation}     
Denote $L=\p S_{1, \phi}$. 
Hence $L$ is $C^k$ smooth ($k\ge 2$) and uniformly convex if and only if  the free boundary $\Gamma$ is. 
It was proved in \cite{S05}  that  $\Gamma$ is uniformly convex and $C^{1,1}$ smooth for $f\equiv 1$.
In the two dimensional case, 
it was proved in \cite{GJM15}  that the curve $L$ is  smooth and analytic.
Moreover, if $u$ is an entire solution to \eqref{MAs0} with $g\equiv 1$, 
namely $u$ is defined in the whole space $\R^n$,  
then $u$ must be rotationally symmetric after an affine transform of coordinates \cite {JX16, Jo55}.
In this paper, we prove

%\begin{equation}\label{classify1}
%u(x)=\int_0^{|x|} (\tau^n+c)^{\frac 1n}d\tau
%\end{equation}
%for some constant $c>0$.

\begin{theorem}\label{thmB}
Let $u$ be a strictly convex solution of \eqref{MAs0}. 
Then in the spherical  coordinates $(\theta, r)$, $u$ is smooth as a function
of  $\theta$ and $r$  if $f$ is smooth, and we have the Taylor expansion
 \vskip-26pt
\beq\label{tur}
u(\theta, r)=r\, {\Small\text{$ \sum_{i=0}^k $}} \phi_i(\theta) r^{ni} + O(r^{1+(k+1)n})
\eeq
for any integer $k\ge 0$, 
where $\phi_i$ are smooth functions of $\theta$.
Moreover, if $f$ is analytic, then $u$ is  analytic in $\theta$.
 \end{theorem}
 
 By the condition $v_0>0$ on $\pom$, we know that $u$ is strictly convex in $\Om^*$ \cite{Ca90a}.
Theorem \ref{thmA} follows from Theorem \ref{thmB} 
by the duality between the free boundary $\Gamma$ and the section $L$.
The Taylor expansion \eqref{tur} reveals a special geometric profile of the solution $u$ at the origin,
namely $\frac {u(\theta, r)}{r} $ is a smooth function of $\theta$ and $r^n$.
%%%  the coefficients of $r^m$ in \eqref{tur} vanish if $m\ne ni+1$ for integers $i$.
%%%% As far as the authors know, such kind of phenomena was first found in  \cite [(1.7)]{LM}
%%%% for the complex Monge-Amp\`ere equation near the boundary.

To prove Theorem \ref{thmB}, 
our goal is to prove that $\frac ur$, as a function of $\theta, r$, is smooth.
Therefore we will make the change
\beq\label{Trans2}
   \zeta(\theta,r) =\frac{u(\theta,r)}{r},\ \ \
   s =r^{\frac n2}.
  \eeq
%%% This approach is different from other ones employed for the regularity of free boundaries.
The proof of Theorems \ref{thmB}   can be divided into three steps.

%%% \begin{remark}
%%% Theorem \ref{thmA} is different from the classical obstacle problem \eqref{Lapob} 
%%% and other fully non-linear uniformly elliptic obstacle problems in the sense 
%%% that there is no singular point on the free boundary $\Gamma$. 
%%%  \end{remark}

\vskip3pt

\noindent {\bf {Step 1}:}  
We first prove that the function $w=u-\phi$ satisfies the growth condition near $0$,
\beq\label{wx}
C_1|x|^{n+1}\le w(x)\le C_2|x|^{n+1}
\eeq
for two positive constants $C_2\ge C_1>0$. 
Estimates \eqref{wx} are built on Pogorelov and  Savin's interior second derivative estimates. 
%%% By the interior second derivative estimates,
%%% Savin \cite{S05} proved the strict convexity and $C^{1,1}$ regularity of the free boundary.
\eqref{wx} is a key estimate in this paper. 
It implies that equation \eqref {MAs1} below is uniformly elliptic.

\vskip3pt

\noindent {\bf {Step 2}:} 
We next prove that the free boundary is $C^2$ smooth.  
We express the solution $u$ in the spherical coordinates $(\theta,r)$,
and make the changes \eqref{Trans2}.
Then   in an orthonormal frame $\theta$ on the unit sphere $\mathbb S^{n-1}$, $\zeta$ satisfies the equation
\begin{equation}\label{MAs1}
\det \begin{pmatrix}
\big(\frac n2\big)^2 \zeta_{ss}+\frac{n(n+2)}{4}\frac{\zeta_s}{s}&   \frac n2\zeta_{s\theta_{1}}&\cdots&  \frac n2\zeta_{s\theta_{n-1}}\\[3pt]
\frac n2\zeta_{s\theta_1}& \zeta_{\theta_1\theta_1}+\zeta+\frac n2 s\zeta_s & \cdots&\zeta_{\theta_1\theta_{n-1}}\\[3pt]
\cdots &\cdots &\cdots &\cdots  \\[3pt]
 \frac n2\zeta_{s\theta_{n-1}} &\zeta_{\theta_1\theta_ {n-1}} & \cdots&    \zeta_{\theta_{n-1}\theta_{n-1}}+\zeta+\frac n2 s\zeta_s
\end{pmatrix}=g.
\end{equation}
This is a fully nonlinear singular elliptic equation. 
By \eqref{wx} we have 
\beq\label{zeta}
 C_1s^2\le \zeta(\theta,s)-\frac{\phi}{r} \le C_2 s^2.  
\eeq
From \eqref{zeta}, we deduce that equation \eqref{MAs1} is uniformly elliptic.
Moreover, the regularity of the free boundary is equivalent to that of the function $\zeta(\theta,0)$. 

Therefore the main task of the paper 
is to establish the regularity, at the boundary $\{s=0\}$, for the fully nonlinear, singular elliptic equation \eqref{MAs1}.
A key step is to obtain the $C^2$ regularity of the solution $\zeta$.
Recall that for the classical obstacle problem \eqref{Lapob},
the key estimate is the $C^1$ regularity, proved in \cite{Ca77}.

The proof of the $C^2$ regularity of $\zeta$ will be carried out as follows. 
We first use a blow-up argument to simplify equation \eqref{MAs1} to the following equation
\begin{equation}\label{MAs2}
\det \begin{pmatrix}
    \psi_{x_nx_n}+\frac{n+2}{n}\frac{\psi_{x_n}}{x_n}&    \psi_{x_nx_1} & \cdots &  \psi_{x_nx_{n-1}}\\[3pt]  
   \psi_{x_nx_1}& \psi_{x_1x_1} & \cdots&\psi_{x_1x_{n-1}}\\[3pt] 
  \cdots &\cdots &\cdots &\cdots  \\[3pt]
  \psi_{x_nx_{n-1}} &\psi_{x_1x_{n-1}} & \cdots&\psi_{x_{n-1}x_{n-1}}
 \end{pmatrix}
    =1 \quad \text{in}\quad \mathbb R^n_+,
\end{equation}
where $\mathbb R^n_+=\Bbb R^n\cap\{x_n>0\}$.
The RHS of \eqref{MAs2} is a positive constant, which we assume is one.

We then make a partial Legendre transform for $\psi$ \cite{LS17}, i.e., 
\beq\label{Par-Leg}
{\begin{split}
  & y_n=x_n,\quad \\
  & y'=D_{x'}\psi,\hskip30pt\\
  & \psi^*=x'\cdot  D_{x'}\psi-\psi,
  \end{split}}
\eeq
where $x'=(x_1, \cdots, x_{n-1})$.
Then
 $\psi^*$ satisfies 
\begin{equation}\label{MAs3}
\psi^*_{y_ny_n}+\frac{n+2}{n}\frac{\psi^*_{y_n}}{y_n}+\det D^2_{y'} \psi^*=0\quad \text{in}\quad \mathbb R^{n}_+.
\end{equation}
A nice feature of equation \eqref{MAs3} is that the singular part $\frac{\psi^*_{y_n}}{y_n}$
is  separate from the nonlinear part of the equation, 
which enables us to prove the $C^{2,\alpha}$ regularity for equation \eqref{MAs3}. 
Hence by a scaling argument, we obtain the following Bernstein theorem.

\begin{theorem}\label{thmC}
Let $\psi\in C^{1,1}(\overline{\mathbb R^n_+})$ be a solution to \eqref{MAs2}.
Assume that $\psi_{x_n}(x',0)=0\ \forall\ x'\in\R^{n-1}$, 
and equation \eqref{MAs2} is uniformly elliptic.
Then 
\begin{equation}\label{quad}
\psi(x)=q(x')+ a x_n^2
\end{equation}
where $a$ is a positive constant and $q$ is a quadratic polynomial.  
\end{theorem}

To prove the $C^2$-continuity of $\zeta(\theta,s)$ at $\mathbb S^{n-1}\times \{s=0\}$, we use a blow-up argument. 
By Theorem \ref{thmC}, the limit of a blow-up sequence is a quadratic polynomial
of the form \eqref{quad}.
The most delicate part of the  paper is to prove the uniqueness of the limit, 
namely the limit is independent of the choice of the blow-up sequences.
The uniqueness of the limit implies that the second derivatives $D^2 \zeta$ are continuous.
Our proof of the uniqueness is by employing the maximum principle consecutively in an infinite sequence of domains.

\vskip5pt

\noindent {\bf {Step 3}:} 
We differentiate equation \eqref{MAs1} to obtain a linearized equation, of which a prototype is of the form
\begin{equation}\label{Lapb}
\Delta_{x'} u+u_{nn}+b\frac{u_n}{x_n} = f\quad \text{in}\quad \mathbb R^n_+  ,
\end{equation}
where  $b>1$, $u_n=u_{x_n}$.
Equation \eqref{Lapb} is a singular elliptic equation of Keldysh type.
In \cite{H95, H96}, Horiuchi introduced the Green function for \eqref{Lapb} and proved the $C^{2,\alpha}$ estimate for \eqref{Lapb}.
In this paper, we use his Green function to 
establish a weighted $W^{2,p}$ estimate for equation \eqref{Lapb}, 
extending the classical $W^{2,p}$ estimate for the Poisson equation.

We then use the freezing coefficient method and the weighted $W^{2,p}$ estimate
to obtain the $C^{2,\alpha}$ regularity of $\zeta$  in $\theta$.
By the weighted $W^{2,p}$ estimate and the bootstrap technique, 
we show that the solution $\zeta\in C^\infty$ up to the boundary.

 Finally we prove that the free boundary is analytic.
The analyticity of solutions has been extensively studied in literature 
\cite{F58, Mo58, KN77}.
A simple proof was found in \cite{Kato96, Bla}.
To prove the analyticity of our free boundary, 
we need to prove the analyticity of $\zeta$ in $\theta$.
We will adopt the method in  \cite{Kato96, Bla}.

The paper is organized as follows. 
In Section 2 we establish the asymptotic estimate \eqref{wx}, 
which is the first step of our proof.
The second step of the proof consists of Sections 3 and 4.
In Section 3 we prove the H\"older continuity for the singular term $\frac{\psi^*_{y_n}}{y_n}$ in equation \eqref{MAs3},
from which we obtain the Bernstein Theorem \ref{thmC}.
In Section 4, we use the Bernstein Theorem and construct auxiliary functions to prove the $C^2$-regularity of the free boundary.
Step 3 of the proof consists of Sections 5 and 6.
In Section 5 we establish a weighted $W^{2,p}$-estimate for \eqref{Lapb}, and use it 
to prove the $C^\infty$ regularity of the free boundary.
The analyticity of the free boundary will be proved in Section 6.

\section{\bf Asymptotic behaviour at the singularity}
%Section 2

Let $u$ be a generalized solution to
\begin{equation}\label{S2.1}
\det D^2 u=g(Du) +\delta_0\quad \text{in}\quad  B_1(0) .
\end{equation}
Assume that $u(0)=0$, $u\ge 0$ and $u>0$ on $\p B_1(0)$.
Assume also that $g$ is a smooth and positive function.
There is no loss of generality in assuming that the unit ball $B_1(0)\subset \Om^*$.
By extending $v$ to $\R^n$ such that $v$ is smooth and uniformly convex away from the free boundary $\Gamma$, 
we may also assume that $\Gamma\subset B_1(0)\subset \Om$. 
Therefore in the following we will consider problems \eqref{MAob} and \eqref{MAs0} in $B_1(0)$.
By the regularity of  the  Monge-Amp\`ere equation, we may also assume that 
$|D^4 u|\le M_0$, $|D^4 v|\le M_0$ near $\p B_1(0)$.
In the following  we will use $\Om$ to denote a general bounded convex domain.

Let 
\beq\label{uwp}
u=w+\phi
\eeq
where $\phi$ is the tangent cone of $u$ at $0$, defined in \eqref{tc}.
Then 
\begin{equation}
w(x)\ge 0,\quad w(x)=o(|x|) \quad \text{as}\quad x\rightarrow 0.
\end{equation}
Let  $\Gamma$ 
be the free boundary of the obstacle problem \eqref{MAob}.
In \cite{S05}, Savin proved that $\Gamma$ is uniformly convex  and $C^{1,1}$ smooth 
for the case $f\equiv 1$. 
In this section, we show that his argument also applies to general function $f(x)$ 
provided $f(x)$ is positive and smooth, and obtain the estimate \eqref{wx}.

\begin{lemma}\label{L2.1} (Pogorelov's estimate)
Let $\psi\in C^4(\Om)$ be a convex solution to 
\beq\label{DP}
{\begin{split}
\det D^2 \psi & =h(x, \psi, D\psi)\ \ \text{in}\ \Om,\\
\psi&=0\ \ \text{on}\ \pom.
 \end{split}}
 \eeq
%{\color{blue} Assume that $\sup |D u|^2\le \delta$ for some small $\delta>0$ depending on $f$}. 
Then there exists a constant $C>0$, depending only  on $n$, $\sup_\Om (|\psi|+|D\psi|)$, and $\|\log h\|_{C^2}$, such that
\beq\label{pog}
[-\psi(x)]  \left| D^2\psi(x)\right|\le C \ \ \forall\ x\in\Om.
\eeq
\end{lemma}

%%% \begin{proof}  
The proof for Pogorelov's estimate can be found in several papers. 
See for instance \cite{GT, Po72}. Here we omit the proof.

\begin{corollary}\label{C2.1} 
Let $u$ be the strictly convex solution to \eqref{S2.1}.
We have the estimate
\beq\label{xdp0}
|x| |D^2 u(x)| \le C \quad \forall \ x\in B_1\backslash\{0\}.
\eeq
where  $C$ depends only on $n, M_0$, $\sup_{B_1} (|u|+|Du|)$,  and $\|\log g\|_{C^2}$.
\end{corollary}

\begin{proof}
Consider a point $x_0\in B_1(0)\backslash\{0\}$ near the origin.  
Choose the coordinates such that $x_0=|x_0| e_n$,
where  $e_n=(0, \cdots, 0, 1)$. 
Subtracting a linear function, one can assume
\begin{equation*}
\phi(te_n)=0\  \text{for}\  t>0, \ \ \text{and}\  \phi(x)\ge 0 \ \text{for}\ x\in B_1.
\end{equation*}
Denote $\Omega_{\eps_0}=\{x\in B_1 \ | \ u(x)<\varepsilon_0 x_n \}$, for some small but fixed $\eps_0>0$.
Applying Lemma \ref{L2.1} to $u-\varepsilon_0 x_n$ in $\Omega_{\eps_0}$, we obtain 
\begin{equation*}
(\varepsilon_0|x_0|-u(x_0))|D^2 u|(x_0)\le C\ \ \forall\ x_0\in\Omega_{\eps_0}.
\end{equation*}
Note that $u(x_0)=o(|x_0|)$. We obtain \eqref{xdp0}.
 \end{proof}

\begin{corollary}\label{C2.2} 
Let $\phi$ be the tangential convex cone of $u$ at $0$. Then there holds
\beq\label{xdp}
|x|\, |D^2 \phi(x)|\le C\ \ \forall\ 0\ne x\in \R^n,
\eeq
for the same constant $C$ in \eqref{xdp0}.
\end{corollary}

\begin{proof}
Consider a point $x_0\in B_1(0)\backslash\{0\}$ near the origin.  
Denote $t=|x_0|\in (0,1)$.
Make the dilation $X=x/t$ and $U_t(X)=u(x)/t$, and
choose the coordinates such that $X_0=:x_0/t=e_n$. 
Then $\phi$ is also the tangential convex cone of $U_t$ at $0$.
By Corollary \ref{C2.1}, $U_t$ satisfies 
\begin{equation*}
|X||D_X^2 U_t(X)|=|x||D^2_x u(x)|\le C\quad \forall \ X\ne 0.
\end{equation*}
Since $U_t(X)\rightarrow \phi(X)$ locally uniformly in $\mathbb R^n$ as $t\rightarrow 0$, 
we obtain \eqref{xdp}. 
\end{proof}
 
Estimates \eqref{xdp0}, \eqref{xdp} are due to Savin \cite{S05}.
By duality, \eqref{xdp} implies that the free boundary $\Gamma$ is strictly convex.
In fact, for any point $p\in\Gamma$, we can choose the coordinates such that $p=0$,
$\Gamma\subset\{x_n\ge 0\}$, and locally $\Gamma$ is given by $x_n=\eta(x')$. 
Then  \eqref{xdp} implies that $\eta(x')\ge  |x'|^2/C$. 
By \eqref{xdp0} and \eqref{xdp}, we also have
\beq\label{xdp1}
|x| |D^2 w(x)| \le C   \quad \forall \ x\in B_1\backslash\{0\}.
\eeq
where $w$ is the function given in \eqref{uwp}.

  \vskip10pt 

\begin{lemma}\label{L2.2} (Savin's estimate)
Consider the Monge-Amp\`ere obstacle problem 
\beq\label{obf}
{\begin{split}
\det D^2 {{v}} & =f(x, {{v}})\chi_{\{{{v}}>0\} }\ \ \text{in}\ \Om,\\
 {{v}}&={{v}}_0>0\ \ \text{on}\ \pom.
 \end{split}}
 \eeq
Assume that  $f$ is non-decreasing in ${{v}}$ and $0<\lambda\le f\le \Lambda<+\infty$. 
Then there exists a constant $C$, depending only  on 
 $n,\lambda,\Lambda,\|D_x \log f\|_{L^\infty}$, $\|D_x^2  \log f\|_{L^\infty}$ and the strict convexity of  $\Gamma$,
such that
\beq\label{pka}
\kappa_i \le C \ \ \forall\ i=1, \cdots, n-1,
\eeq
where  $(\kappa_1, \cdots, \kappa_{n-1})$ are the principal curvatures  of $\p \{{{v}}=h\}$, for $h> 0$ small.
\end{lemma}

\begin{proof}
The following proof is due to Savin's \cite{S05}, where he considered the case $f=f({{v}})$.
His proof also applies to the case $f=f(x, {{v}})$,  so we will just sketch the proof.
%+% The only difference is that when $f=f({{v}})$, 
%+% one has $\p_{x_i} \bar f=0$ ($1\le i\le n-1$) and $\p_{{\bf{v}}}\bar f=0$ for the function $\bar f$ in \eqref{obg}.
%+% For convenience of the readers, we  include the details here.

1). Let $p_0$ be a point on the free boundary $\Gamma$. 
We choose the coordinates such that $x_n$ is the  outer normal of $\{{{v}}=0\}$ at $p_0$,
and
express the graph of ${{v}}$ by a function ${\bf{v}}$ in the form $x_n= - {\bf{v}}(x_1, \cdots, x_{n-1}, x_{n+1})$.
Then
$$\frac {\det D^2 {\bf{v}}} {(1+|D{\bf{v}}|^2)^{(n+2)/2}} =K 
 =\frac {\det D^2 {{v}}} {(1+|D{{v}}|^2)^{(n+2)/2}} = \frac {f(x, {{v}})} {(1+|D{{v}}|^2)^{(n+2)/2}}, $$
where $K$ is the Gauss curvature of the graph of ${{v}}$.
Hence ${\bf{v}}$ satisfies the equation
$$\det D^2 {\bf{v}}  = f(x, {{v}}) \frac {{(1+|D{\bf{v}}|^2)^{(n+2)/2}} } {(1+|D{{v}}|^2)^{(n+2)/2}} . $$
From the relation 
$x_{n+1}={{v}}(x_1, \cdots, x_{n-1}, -{\bf{v}}(x_1, \cdots, x_{n-1}, x_{n+1}))$,
we have
$$ {{v}}_1-{{v}}_n{\bf{v}}_1=0,\ \ \cdots,\ \  {{v}}_{n-1}-{{v}}_n{\bf{v}}_{n-1}=0, \ \ {{v}}_n {\bf{v}}_{n+1}=-1. $$
Hence 
$ D{{v}}=(\frac{-{\bf{v}}_1}{{\bf{v}}_{n+1}}, \cdots,  \frac{-{\bf{v}}_{n-1}}{{\bf{v}}_{n+1}}, \frac {-1}{{\bf{v}}_{n+1}}) $,
and
$ \frac {(1+|D{\bf{v}}|^2)^{(n+2)/2}} {(1+|D{{v}}|^2)^{(n+2)/2}} = |{\bf{v}}_{n+1}|^{n+2} $.
The above equation   becomes
\beq\label{obg}
\det D^2 {\bf{v}}  =  {\bar f (x, {\bf{v}})}  |{\bf{v}}_{n+1}|^{n+2} , 
\eeq
where $\bar f (x, {\bf{v}})=f(x_1, \cdots, x_{n-1}, -{\bf{v}}, x_{n+1})$.
The variables of ${\bf{v}}$ are $x_1, \cdots, x_{n-1}, x_{n+1}$.

By a translation of the coordinates, assume that $p_0=a e_n$ for a small constant $a>0$. 
By the strict convexity of $\Gamma$,  there exists a small constant $\delta_0>0$ such that ${\bf{v}}(x_1,\cdots,x_{n-1},x_{n+1})$ is a graph in $(-2\delta_0,2\delta_0)^{n-1}\times (0,2\delta_0)$, 
${\bf{v}}_1$ is bounded in $(-\delta_0,\delta_0)^{n-1}\times (0,\delta_0)$, ${\bf{v}}(0)<0$, 
and ${\bf{v}}>0$ on $\partial((-\delta_0,\delta_0)^{n-1})\times [0,\delta_0]$.

2). To prove \eqref{pka}, it suffices to prove that ${\bf{v}}_{11}$ is bounded near $p_0$. 
The same argument applies to ${\bf{v}}_{ii}$  for $2\le i\le n-1$.
Let 
\begin{equation*}
G=(-{\bf{v}}) \eta(\frac 12 {\bf{v}}_1^2){\bf{v}}_{11} . 
\end{equation*}
be the auxiliary function of Pogorelov.
Assume that $G$ attains its maximum at  an {\it interior point} $x_0\in \{{\bf{v}}<0\}\cap\{ (-\delta_0,\delta_0)^{n-1}\times (0,\delta_0)\}$.
We make the coordinate transform
\beq\label{cha}
y_1=x_1+\frac{{\bf{v}}_{1i}(x_0)}{{\bf{v}}_{11}(x_0)} x_i,\ \ \
y_i=x_i\   (i=2, \cdots, n-1),\ \ \text{and}\  y_n=x_{n+1},
\eeq
such that $D^2 {\bf{v}}$ is diagonal at $x_0$. 
This coordinate transform does not change the value of ${\bf{v}}_1,{\bf{v}}_{11}$.
Then at $x_0$, we have 
$$(\log G)_i =0\ \ \text{and}\ \  {\Small\text{$\sum^n_{i=1}$}} {\bf{v}}^{ii}(\log G)_{ii} \le 0.$$
Carrying out the calculation in \cite{S05} and choosing $\eta(t)=e^{\alpha t}$ for some $\alpha>0$ small,
we obtain  $G\le C$.

3). In step 2), we assume that $x_0$ is an interior point.
Next we use approximation  \cite{S05}  
to rule out the case when $x_0 \in\{x_{n+1}=0\}$ is a boundary point.
Consider 
\beq\label{obfe}
{\begin{split}
\det D^2 {{v}} & =f_\eps (x, {{v}})\ \ \text{in}\ \Om,\\
 {{v}}&={{v}}_0>0\ \ \text{on}\ \pom
 \end{split}}
 \eeq
where 
$f_\eps (x, {{v}})= a_\eps ({{v}}) f(x, {{v}}) $.
We choose smooth functions $a_\eps$  such that
$$
 a_\eps (t)=\Big\{
 {\begin{split} 
  &\eps\ \ \ \text{when}\ \ t<-\eps, \\[-5pt]
  &1 \ \ \ \text{when}\ \ t>0 .
  \end{split}}$$ 

%% $$a_\eps (t)\ge 0, \ \ \ b_\eps (t)\ge 0, \ \ \ a_\eps (t)+b_\eps (t)=1\ \forall\ t>0.$$
Let ${{v}}_\eps$ be the solution to \eqref{obfe}.
By assumption, $f(x, {{v}})$ is non-decreasing in ${{v}}$, there is a unique solution to the obstacle problem.
Hence ${{v}}_\eps$ converges to the unique solution ${{v}}$ as $\eps\to 0$.

Define ${\bf{v}}_\eps$ similarly as above, such that
$x_{n+1}={{v}}_\eps(x_1, \cdots, x_{n-1}, -{\bf{v}}_\eps)$.
We then apply the above argument to ${\bf{v}}_\eps$ and obtain an upper bound for $({ {\bf{v}}_\eps})_{11}$ near the origin,
and the upper bound is independent of $\eps$. Sending $\eps\to 0$, we obtain \eqref{pka}.
 \end{proof}

%%  \begin{remark} %%% Remark 2.2
%% The proof of Lemmas \ref{L2.2} is more complicated than that of Lemma \ref{L2.1}.
%% In the proof of Lemmas \ref{L2.2}, we make the coordinate change \eqref{cha},
%% where there is no upper bound for the fraction $\frac{v_{1i}(x_0)}{v_{11}(x_0)}$.
%% In the proof of Lemma \ref{L2.1}, we can make a rotation of coordinates such that $D^2u(x_0)$
%% is diagonal, avoiding the fraction $\frac{v_{1i}(x_0)}{v_{11}(x_0)}$ in the linear transform \eqref{cha}. 
%% Note that in equation \eqref{obg}, the function $g$ on the right hand side may be dis-continuous in $x_{n+1}$,
%% but in equation \eqref{DP}, we assume that the function on the RHS is smooth in all directions.
%% Lemma \ref{L2.1} we obtain estimate for all second derivatives
%% but in Lemma \ref{L2.2}  we can only obtain estimate for $v_{ii}$ for $i\le n-1$ but not for $v_{x_{n+1}x_{n+1}}$.
%% \end{remark} %%% Remark 2.1 

\begin{corollary}\label{C2.3} 
Let $u$ be the  solution to \eqref{S2.1}.
We have the estimate
\beq\label{xdp2}
|x| |\p_\xi^2 u(x)| \ge C_1 \ \ \forall\  0\ne x\in B_1(0), \ \xi\perp x, \  |\xi|=1
\eeq
where the constant $C_1$ depends only on $n, M_0$, $\sup_{B_1} (|u|+|Du|)$, and $\|\log g\|_{C^{1,1}}$.
\end{corollary}

\begin{proof}
Denote by $\ell_x$  the tangent plane of $u$ at $x$. 
For any given point $x_0\ne 0$ in $B_1(0)$ near the origin,
let $h  = \ell_{x_0}(0)$ ($h<0$) and $\phi_h(x)$ be the convex cone given by
$$\phi_h(x)=\sup\{ L(x):\ L\ \text{is affine function,}\ L(\cdot) < u(\cdot)\ \text{in}\ B_1(0)\ \text{and}\ L(0)= h\} . $$
Let $v$ be Legendre transform of $u$, i.e. $v=x\cdot D u-u$. Then 
$v$ is a solution to  \eqref{MAob}. By the duality between $u$ and $v$, we have $\p \phi_h \{0\}= \{v\le -h\}$.
So \eqref{pka} implies that 
\beq\label{xdp3}
|x| |\p_\xi^2 \phi_h (x)| \ge C_1 \ \ \forall \ x \ne 0, \ \xi\perp x, \ 	 |\xi|=1.
\eeq
It is easy to see that \eqref{xdp3} implies \eqref{xdp2} at point $x_0$.
\end{proof}

Sending $h\to 0$, from \eqref{xdp3} we also obtain

\begin{corollary}\label{C2.4} 
Let $\phi$ be the tangential cone of $u$ at $0$. Then
\beq\label{xdp4}
|x| |\p_\xi^2 \phi (x)| \ge C_1 \ \ \forall\ 0\ne x\in B_1(0),\ \xi\perp x,\ 	|\xi|=1.
\eeq
\end{corollary}
 
We denote $a\approx b$ if two quantities $a$ and $b$ are positive and there is a constant $C$ under control 
such that $\frac ab + \frac ba \le C$. Given two convex domains $A$ and $B$, we  denote
$A\sim B$ if $C^{-1}(A-p)\subset B-q\subset C(A-p)$, where $p, q$ are the geometric centres of $A$ and $B$,
respectively.

\begin{corollary}\label{C2.5} 
Let $u$ be the  solution to \eqref{S2.1}.
Let $\lambda_1(x)\ge \cdots \ge\lambda_n(x)$ be the eigenvalues of $D^2 u$ at $x\ne 0$. 
Then
$\lambda_1(x)\approx \cdots \approx\lambda_{n-1}(x)\approx |x|^{-1}$ and 
$\lambda_n(x)\approx |x|^{n-1}$.  
\end{corollary}

\begin{proof}
By \eqref{xdp0} and \eqref{xdp2}, we have $\lambda_1\approx \cdots \approx\lambda_{n-1}\approx |x|^{-1}$.
By virtue of  the equation \eqref{S2.1} we then have $\lambda_n\approx |x|^{n-1}$.  
\end{proof}

\begin{lemma}\label{L2.3} 
Let $w$ be the function given in \eqref{uwp}.
There exist two constants $C_1,C_2>0$,
depending only on $n, M_0$, $\sup_{B_1} (|u|+|Du|)$, and $\|\log g\|_{C^{1,1}}$,
 such that 
\begin{equation}\label{asyw}
C_1|x|^{n+1}\le w(x)\le C_2|x|^{n+1}\ \ \forall\ x\in B_1(0). 
\end{equation}
\end{lemma}

\begin{proof}
For any given point $x_0\in B_1(0)\backslash\{0\}$ near the origin,
by a rotation of the coordinates we assume that $x_0$ is on the positive $x_1$-axis.
Subtracting a linear function we assume 
\begin{equation}\label{phit}
{\begin{split}
 & \phi(te_1)=0\ \ \forall\  t\ge 0,\\
  & \phi(x)\ge 0\ \  \forall \ x\in B_1(0),
  \end{split}}
\end{equation}
 where $e_1=(1, 0, \cdots, 0)$.
To prove the first inequality of \eqref{asyw},  it suffices to prove 
\begin{equation*}
u(te_1)\ge Ct^{n+1}\ \ \ \text{for $t>0$ small}.  
\end{equation*}
By Corollary \ref{C2.5}, we have 
$$u_{11}(se_1)\ge Cs^{n-1}.$$
Hence 
$$u_1(te_1)=\int_0^t u_{11}(s e_1) ds\ge Ct^n , $$
 and so we have
$$u(te_1)=\int_0^t u_1(se_1)ds\ge Ct^{n+1}.$$

Next we prove the second inequality of \eqref{asyw}.  Similarly as in Corollary \ref{C2.3}, denote by $\ell_x$  the tangent plane of $u$ at $x$. 
For any given point $te_1$, $t>0$, 
let $\phi_t(x)$ be the convex cone, given by
$$\phi_t(x)=\sup\{ L(x):\ L\ \text{is affine function,}\ L(\cdot) < u(\cdot)\ \text{in}\ B_1(0)\ \text{and}\ L(0)= \ell_{te_1}(0)\} . $$
Then from the proof of Corollary \ref{C2.3}, we have
\begin{equation*}
\partial_{x_1}^2\phi_t(te_1)=0,\quad \partial_{x_i}^2\phi_t(te_1)\ge \frac{C}{t},\quad i=2,\cdots,n.
\end{equation*}
Let $\xi$ be a unit  eigenvector of  the least eigenvalue  of $D^2 u(te_1)$. 
We claim
\beq\label{claim}
|\lan \xi,e_1\ran|\ge 1-C_1 t^{n} 
\eeq
for a sufficiently large constant $C_1$.
Indeed, 
let $\xi=a e_1+b\eta$, where $\eta\perp e_1$ and $|\eta|=1$.  
If \eqref{claim} is not true, then $|b| \ge  (C_1 t^n)^{1/2}$ for $t$ small enough. 
Since $u$ and $\phi_t$ are tangent at $te_1$ and $u\ge \phi_t$, 
and since $\phi_t(te_1)$ is linear in $t$, we have
\begin{equation*}
t^{n-1} \approx \partial_{\xi}^2u(te_1)\ge \partial_{\xi}^2\phi_t(te_1)
=b^2 \partial_{\eta}^2\phi_t(te_1)\ge CC_1t^{n-1}
\end{equation*}
which is impossible if we choose $C_1$ large enough. This proves the claim.

By \eqref{claim}, we infer that 
\begin{equation}
\partial_{x_1}^2u(te_1)\le C_2 t^{n-1}.
\end{equation} 
Indeed, from \eqref{claim}, we have $a\approx 1$ and  $|b|\le Ct^{\frac n2}$.  By the convexity of $u$, we have
\begin{equation}
\begin{split}
t^{n-1} \approx \partial_{\xi}^2u(te_1)=&a^2 \partial_{x_1}^2 u(te_1)+2ab \partial_{x_1\eta}u(te_1)+b^2 \partial_\eta^2 u(te_1)\\
\ge&  \Big(a\sqrt{\partial_{x_1}^2 u(te_1)}- |b| \sqrt{\partial_\eta^2 u(te_1)}\, \Big)^2\\
\ge &  a^2\Big( \sqrt{\partial_{x_1}^2 u(te_1)}-Ct^{\frac{n-1}{2}}  \Big)^2 
\ge \frac{C_2}4 t^{n-1}
\end{split}
\end{equation}
if $\partial_{x_1}^2u(te_1)\ge C_2 t^{n-1}$ for a sufficiently large $C_2$, 
which is  a contradiction.
Hence 
$$u_1(te_1)=\int_0^t u_{11}(s e_1) ds\le C_3t^n$$
 and so we have
$$u(te_1)=\int_0^t u_1(se_1)ds\le C_4t^{n+1}.$$
This finishes the proof.
\end{proof}

%\begin{remark}  %%% Remark 2.3
%By \eqref{asyw} one easily sees that the constant $t_\eps$ satisfies
%\beq\label{ete} 
%C^{-1}\eps^{1/n} \le t_\eps\le  C\eps^{1/n} .
%\eeq
%\end{remark}  

We express equation \eqref{S2.1} in the spherical coordinates $(\theta,r)$,
where $r=|x|$ and $\theta=(\theta_1, \cdots, \theta_{n-1})$ is  an orthonormal frame on $\mathbb S^{n-1}$.
 
\begin{lemma}\label{L2.4} 
In the spherical coordinate $(\theta,r)$, one has
\begin{equation} \label{upb}
{\begin{split}
  &  |\p_r^k w(p)|\le c|p|^{n+1-k},\quad k=0,\cdots,n+1,\\
   &  |\p_{\theta}^2 w(p) |\le c|p|,\\
  & |\p_{r\theta}w(p)|\le c|p|^{\frac n2} , 
  \end{split}}
\end{equation}
for any point $p\ne 0$ near the origin.
%\centerline {\Small\color{red}
%For uniform ellipticity, do we need to prove that $ \|\p_r^2 w(r,\theta)\|\approx r^{n-1}$ for all $r, \theta$? }
%\centerline {\Small\color{red}
%From \eqref{po2} it seems  \eqref{upb} is enough.}

\end{lemma}

\begin{proof}
As in Lemma \ref{L2.3}, we may assume \eqref{phit} holds.
Denote
$G_\ve=\{x\in B_1(0): \ u(x)<\ve x_1\}$, 
where $\eps>0$ is a small constant. 
By the strict convexity of $u$, we have $G_\ve \subset\subset B_1$.

For any point $x=(x_1, \tilde x)\in G_\eps$, where $\tilde x=(x_2, \cdots, x_n)$,
by \eqref{xdp4} we have
\begin{equation}\label{uphi}
u(x)\ge \phi(x)\ge  c_0\frac{|\tilde x|^2}{x_1}.
\end{equation}
Hence
\begin{equation}\label{101}
G_\ve\subset\{x\in B_1(0): \ |\tilde x|\le c_1\ve^{\frac 12} x_1\}.
\end{equation}
Denote
\beq \label{te}
{\begin{split}
s_\ve & =\sup\{s: \ se_1\in G_\ve\},\quad \\
 t_\varepsilon & =\sup\{x_1: \ x\in G_\varepsilon\}. 
 \end{split}} 
\eeq
By Lemma \ref{L2.3}, we have $t_\varepsilon\ge s_\varepsilon \approx \varepsilon^{\frac 1n}$. 

By the definition of $t_\eps$ in \eqref{te}, and the strict convexity of $u$, 
there exists a unique $\tilde x_\varepsilon$ such that 
$(t_\varepsilon, \tilde x_\varepsilon)\in \partial G_\varepsilon$. Then by Lemma \ref{L2.3},
\begin{equation*}
\begin{split}
\varepsilon t_\varepsilon=u(t_\varepsilon, \tilde x_\varepsilon)
 \ge \phi(t_\varepsilon, \tilde x_\varepsilon)+C(t^2_\varepsilon+|\tilde x_\varepsilon|^2)^{\frac {n+1}{2}}\ge Ct_\varepsilon^{n+1},
\end{split}
\end{equation*}
which implies $t_\eps\le C\eps^{\frac 1n}$.
Hence $t_\varepsilon\approx s_\varepsilon\approx \varepsilon^{\frac 1n}$.

By \eqref{asyw}, we also have  
\begin{equation*}
\inf_{G_\varepsilon} (u-\varepsilon x_1)
 =  \inf_{G_\varepsilon} (\phi+w -\varepsilon x_1) 
\approx -\varepsilon s_\varepsilon.
\end{equation*}
Let $A_\eps = G_\eps\cap\{x_1=\beta s_\eps\}$, where $\beta>0$ is a small constant. By Corollary \ref{C2.4}, we have
$A_\eps\subset \{|\tilde x|<C\eps^{1/2}s_\eps\}\cap\{x_1=\beta s_\eps\}$, where $\tilde x=(x_2, \cdots, x_n)$.
For a point  $(x_1, \tilde x)\in\p A_\eps$,
by Corollary \ref{C2.2} and Lemma \ref{L2.3}, we have
$$ 
u(x_1, \tilde x)
   \le \phi(x_1, \tilde x)+C\big(|x_1|^2+|\tilde x|^2\big)^{\frac{n+1}{2}}
  \le C{\Small\text{$\frac{|\tilde x|^2}{x_1}$}} + 2 C \beta^{n+1}s_\varepsilon^{n+1} . $$
The left hand side $u(x_1, \tilde x)=\eps x_1=\beta \eps s_\eps$  since $(x_1, \tilde x)\in\p A_\eps$.
Choosing a small $\beta<<1$  such that
$ \beta^{n+1}s_\varepsilon^{n+1} \approx \beta^{n+1}\eps s_\eps<<\beta \eps s_\eps$, 
we obtain $|\tilde x|\ge c\eps^{1/2} s_\eps$, namely $ \{|\tilde x|<c\eps^{1/2}s_\eps\}\cap\{x_1=\beta s_\eps\} \subset A_\eps$.

Now we make the coordinate change $x\rightarrow y=T(x)$, given by
\begin{equation}
y_1=\frac{x_1}{s_\varepsilon},\quad y_k=\frac{x_k}{\varepsilon^{\frac 12}s_\varepsilon}\ \ \ (k=2,\cdots,n),
\end{equation}
We have shown that $A_\eps\sim  \{|\tilde x|<\eps^{1/2}s_\eps\}$ (as a convex domain in $\R^{n-1}$).
Hence $T(G_\varepsilon)$ has a good shape, namely,
$ T(G_\varepsilon)\sim B_1(0)$.
Let
$ \tilde u(y)=\frac{u(x)-\varepsilon x_1}{\varepsilon s_\varepsilon}. $
Then $\tilde u$ satisfies
\begin{equation}
{\begin{split}
\det D^2 \tilde u &=c_\varepsilon\tilde g(D\tilde u) \quad \text{in}\quad T(G_\varepsilon),\\
\tilde u &=0 \quad \ \ \ \text{on}\quad \partial T(G_\varepsilon) ,
\end{split}}
\end{equation}
where $c_\varepsilon=\eps^{-1} s_\varepsilon^n\approx 1$, 
$\tilde g(D\tilde u)=g(\varepsilon \tilde u_{y_1}+\varepsilon,\varepsilon^{\frac 12}\tilde u_{\tilde y})$. 
Hence by the regularity theory for the  Monge-Amp\`ere equation,  $\tilde u$ is smooth in $T(G_\varepsilon)$,
and we have the estimate
\begin{equation}\label{rupb}
\| D^k \tilde u\|_{S_{h_0/2,\tilde u}}\le C_k,\quad k\ge 2
\end{equation}
where $S_{{h_0/2},\tilde u}=\{\tilde u<-h_0/2\}$, $h_0=: - \tilde u(\frac 12, 0)\approx 1$. 
Restricting to the $x_1$-axis,  we obtain
\begin{equation*}
\|D^k_{y_1}\tilde w\|_{S_{h_0/2,\tilde u}\cap\{|\tilde x|=0\} }
 = \|D^k_{y_1}\tilde u\|_{S_{h_0/2,\tilde u}\cap\{|\tilde x|=0\} } 
 \le C_k,\quad k\ge 2,
\end{equation*}
where $\tilde w(y)=\frac{w(x)}{\varepsilon s_\varepsilon}$.
Scaling back to the original coordinates, we obtain
\begin{equation*}
|D^k_{x_1} w(p)|\le C_k t_{\varepsilon}^{n+1-k},\quad k\ge 2 , 
\end{equation*}
where $p=te_1$ with $t\approx t_\varepsilon$. 
We obtain the first estimate in \eqref{upb}.
The second and third estimates in \eqref{upb}, for which $k=2$, also follows from \eqref{rupb} by rescaling.
\end{proof}

 \vskip10pt

\section{\bf Bernstein theorem for a singular Monge-Amp\`ere equation}
%Section 3

In this section we prove a Bernstein theorem for the singular Monge-Amp\`ere type equation in half space,
\begin{equation}\label{blow1}
\det \begin{pmatrix}
   \psi_{x_nx_n}+b\frac{\psi_{x_n}}{x_n}&    \psi_{x_nx_1}  & \cdots &   \psi_{x_nx_{n-1}}\\[3pt] 
   \psi_{x_nx_1}& \psi_{x_1x_1} & \cdots&\psi_{x_1x_{n-1}}\\[3pt] 
   \cdots &\cdots &\cdots &\cdots \\[3pt]
  \psi_{x_nx_{n-1}} &\psi_{x_1x_{n-1}} & \cdots&\psi_{x_{n-1}x_{n-1}}
 \end{pmatrix}=1\quad \text{in}\quad \mathbb R^n_+=\Bbb R^n\cap\{x_n>0\} .
\end{equation}
%We always assume that $b$ is a positive constant. {\color{blue} ($b>1$?)}
Equation \eqref{blow1} is the limit of equation \eqref{MAs1} in a blow-up argument.
We have the following Bernstein theorem.

\begin{theorem}\label{thmbern}
Let $\psi\in C^{1,1}(\overline{\mathbb R^n_+})$ be a solution to \eqref{blow1}   with constant $b>-1$.
Assume that $D \psi(0)=0$, $\psi_{x_n}(x',0)=0\ \forall\ x'\in\R^{n-1}$, 
and equation \eqref{blow1} is uniformly elliptic.
Then $\psi$ is a quadratic polynomial of the form
\begin{equation}\label{qp1}
\psi(x)=\frac 12 {\Small\text{$ \sum_{i,j=1}^{n-1} $}} c_{ij} x_ix_j+\frac{1}{2}c_{nn}x_n^2
\end{equation}
where $\{c_{ij}\}_{i,j=1}^{n-1}$ is positive definite and $c_{nn}>0$.
\end{theorem}

%% {\color{blue}
%% \noindent{\bf Question:} 
%% Without the uniform ellipticity condition, is there counter example to Theorem \ref{thmbern}?  }

In Theorem \ref{thmbern}, we denote $\overline{\mathbb R^n_+}=\mathbb R^n_+\cup \{x_n=0\}$.
To prove the Bernstein theorem, we make use of the H\"older continuity for the following degenerate elliptic equation.
\begin{equation}\label{aeq-Deg0}
\partial_n(x_n\partial_n u)+{\Small\text{$\sum_{i,j=1}^{n-1}$}}\p_i(a_{ij}(x)\p_ju)
         +{\Small\text{$\sum_{i=1}^n$}}b_i(x)\p_iu=f(x)\quad\text{in }\mathbb R^n_+.
\end{equation}
We assume that the coefficients $a_{ij}$ and $b_i$ satisfy the following conditions.
\begin{itemize}
\item[(i)]   $a_{ij}$ are measurable and satisfy
\begin{equation*} %%% \label{con(i)}
C_*^{-1} |\xi|^2\le a_{ij}\xi_i\xi_j\le C_* |\xi|^2
\quad\forall\ \xi \in\mathbb R^{n-1}, 
\end{equation*}
where $C_*$ is a positive constant. 

\item[(ii)]  $b_1=\cdots=b_{n-1}=0$ and   $b_n$ is a positive constant.
\end{itemize}

%Note that condition (i) implies 
%\begin{equation*}
% {C_*}^{-1} x_n\le a_{nn}  \le C_* x_n,\ \ \ 
% |a_{in}| \le C\sqrt{x_n},\quad i=1,\cdots,n-1. 
%\end{equation*}
By a change of variables,  a model of equation \eqref{aeq-Deg0} is
\begin{equation}\label{aeq-Deg1} 
\Delta u+\frac {b}{x_n} u_{x_n} =f(x)\quad\text{in }\mathbb R^n_+.
\end{equation}
This is the classical Keldysh equation.

Denote $B^+_R(0) = B_R(0)\cap \{x_n>0\}$.
We have the following H\"older continuity.

\begin{proposition}\label{amoser2b}
Let $u\in C^2(B_1^+)\cap L^\infty(B_1^+)$ be a solution to \eqref{aeq-Deg0}.
Assume conditions (i), (ii), and $f\in L^q(B_1^+)$ for some $q>(n+1)/2$. 
Then $u$ is continuous up to $x_n=0$, and there exists $\alpha\in (0,1)$ such that
\begin{equation}\label{a418z}
     |u(x)-u(\tilde x)|
     \le C\big(\sup_{B_1^+}|u|+\|f\|_{L^q(B_1^+)}\big) |x-\tilde x|^{\alpha}\ \ \ \forall\ x,\tilde x\in  B_{1/2}^+ ,
\end{equation}
where $\alpha$ and $C$ are positive constants depending only on
$n, b, q, C_*$.
\end{proposition}

The H\"older continuity of solutions for degenerate elliptic equations has been studied by many authors.
For proofs of Proposition \ref{amoser2b}, we refer the readers to  \cite{FP, HHH}.

In Theorem 1.11 of \cite{FP}, 
the authors proved the H\"older continuity for weak solutions to the variational equation (1.19) in \cite{FP},
where the bilinear $a(u, v)$ was given in (1.13) and the coefficients satisfy Assumption 1.11 of the paper.
In Section 5 of \cite{HHH}, the authors studied equation \eqref{aeq-Deg0} 
and proved the H\"older continuity in dimension two, with application to a geometric problem in $\R^3$, 
but the proof in \cite{HHH} is valid in all dimensions. 
In fact, the proofs in \cite{FP} and \cite{HHH} are similar,
both of them use the Nash-Moser iteration. 
 %%% A probability proof of Proposition \ref{amoser2b} was given in \cite{ZD}.

In  \cite{FP, HHH} the coefficients $b_i$ can be more general,
here we assume condition (ii), which suffices for our purpose.
In Proposition \ref{amoser2b} we also assume that $u\in C^2(B_1^+)\cap L^\infty(B_1^+)$,
which is stronger than the assumption that $u$ is a weak solution in  \cite{FP, HHH}.

\vskip10pt

To apply Proposition \ref{amoser2b} to the singular Monge-Amp\`ere equation \eqref{blow1},
we make a  partial Legendre transform \cite{LS17}, 
to change equation  \eqref{blow1} to the form \eqref{aeq-Deg0}.

Let 
\begin{equation} \label{plt}
{\begin{split}
  & y_n=x_n,\quad \\
  & y'=D_{x'}\psi,\hskip30pt\\
  & \psi^*=x'\cdot D_{x'}\psi-\psi.
  \end{split}}
\end{equation}
Then by direct computation, we have
\begin{equation*}
{\begin{split}
 &\frac{\p y_n}{\p x_n}=1,\quad \frac{\p y_n}{\p x'}=0,\\
 &\frac{\p y'}{\p x_n}=D_{x'} \psi_{x_n},\quad \frac{\p y'}{\p x'}=D^2_{x'}\psi,
   \end{split}}
\end{equation*}
and
\begin{equation*}
{\begin{split}
  & \frac{\p x_n}{\p y_n}=1,\quad \frac{\p x_n}{\p y'}=0,\\
  &  \frac{\p x'}{\p y'}=(D^2_{x'}\psi)^{-1},\\
 \hskip80pt &  \frac{\p x_i}{\p y_n}=\Psi^{in}(\det (D^2_{x'} \psi))^{-1}, \quad i=1,\cdots,n-1
     \end{split}}
\end{equation*}
where $\{\Psi^{ij}\}$ is the cofactor matrix of $D^2 \psi$. 
Hence $\psi^*$ satisfies
\begin{equation*}
{\begin{split}
   & \psi_{y_n}^*=-\psi_{x_n},\quad D_{y'} \psi^*= x',\\
   & D^2_{y'} \psi^*=(D^2_{x'} \psi)^{-1} ,
        \end{split}}
\end{equation*}
and 
\begin{equation*}
{\begin{split}
\psi^{*}_{y_ny_n} 
   &=-\frac{\p \psi_{x_n}}{\p y_n}\\
   & =-\psi_{x_nx_n} - {\Small\text{$ \sum_{i=1}^{n-1} $}} \, \frac{\p \psi_{x_n}}{\p x_i}\frac{\p x_i}{\p y_n}\\
   & =-\psi_{x_nx_n} - {\Small\text{$ \sum_{i=1}^{n-1} $}} \, \psi_{x_ix_n}\Psi^{in}(\det D^2_{x'} \psi )^{-1} .
           \end{split}}
\end{equation*}
We obtain
\begin{equation*}
-\frac{\psi^{*}_{y_ny_n}+b\frac{\psi^*_{y_n}}{y_n}}{\det D^2_{y'} \psi^*}=\psi_{x_ix_n}\Psi^{in}+(\psi_{x_nx_n}+b\frac{\psi_{x_n}}{x_n})\det D^2_{x'}\psi=1.
\end{equation*}
Hence $\psi^*$ satisfies 
\begin{equation}\label{002}
\psi^*_{y_ny_n}+b\frac{\psi^*_{y_n}}{y_n}+\det D^2_{y'} \psi^*=0 \quad \text{in}\quad \mathbb R^{n}_+.
\end{equation}

In equation \eqref{002}, the singular term $\frac{\psi^*_{y_n}}{y_n}$ is  separate from the nonlinear part $\det D^2_{y'} \psi^*$. 
This is a very helpful property. 
Moreover, the Monge-Amp\`ere operator $\det D^2_{y'} \psi^*$ is of divergence form.
Hence equation \eqref{002} is of the same form as \eqref{aeq-Deg1}.
Moreover, we assume that $\psi\in C^{1,1}$ such that \eqref{blow1} is uniformly elliptic.
We have the following key estimate.

\begin{lemma}\label{lemholder}
Let $\psi^*\in C^{1,1}(\overline{\mathbb R^{n}_+})$ be a solution to  \eqref{002} with $b>-1$.
Assume $\psi_{y_n}^*(y',0)=0\ \forall\ y'\in\R^{n-1}$, and $D_{y'}^2 \psi^*$ is positive definite. 
Then $\frac{\psi^*_{y_n}}{y_n}\in C^\alpha(\overline{\mathbb R^n_+})$ for some $\alpha\in (0,1)$, and 
we have the estimate
\beq \label{HE}
\Big\|\frac{\psi^*_{y_n}}{y_n}\Big\|_{C^\alpha(\mathbb R^{n-1}\times [0,1])} \le C
\eeq
for a constant $C$  depending only on $b, n$, 
$\|D_y^2 \psi^*\|_{L^\infty(\mathbb R^n_+)}$ and $\|(D^2_{y'}\psi^*)^{-1}\|_{L^\infty(\mathbb R^n_+)}$. 
\end{lemma}

\begin{proof}
Let $z_n=\frac 14 {y_n^2}$, $z'=y'$. Then equation \eqref{002} is changed to 
\begin{equation*}
z_n\psi^*_{z_nz_n}+\frac{b+1}{2} \psi^*_{z_n}+\det D^2_{z'} \psi^*=0 \quad \text{in}\quad \mathbb R^n_+.
\end{equation*}
Denote $ \Psi=\psi^*_{z_n}$. Differentiating the above equation in $z_n$ gives
\begin{equation*}
\p_{z_n}(z_n\Psi_{z_n})+  {\Small\text{$  \sum_{i,j=1}^{n-1} $}} \p_{z_i} (a^{ij} \Psi_{z_j})+\frac{b+1}{2} \Psi_{z_n}=0
 \quad \text{in}\quad \mathbb R^n_+.
\end{equation*}
Here $\{a^{ij}\}_{i, j=1}^{n-1}$ is the the cofactor matrix of $D^2_{z'} \psi^*$. 
By assumption,  $D_{y'}^2 \psi^*$ is positive definite. Hence  
$
\lambda I\le \{a^{ij}\} \le \Lambda I
$
for two positive constants $\lambda,\Lambda$ 
depending only on $\|D_y^2 \psi^*\|_{L^\infty(\mathbb R^n_+)}$ and $\|(D^2_{y'}\psi^*)^{-1}\|_{L^\infty(\mathbb R^n_+)}$. 
Moreover. 
\begin{equation}\label{ppsi}
\begin{split}
\Psi(z)=\psi^*_{z_n}=\frac{2\psi^*_{y_n}}{y_n}=2\int_{0}^{1}\psi^*_{y_ny_n}(y', ty_n)dt\in L^\infty(\mathbb R^n_+) .
\end{split}
\end{equation}
Therefore all the conditions in Proposition \ref{amoser2b} are satisfied. 

By Proposition \ref{amoser2b}, we obtain the H\"older continuity of $\Psi$.
By \eqref{ppsi},  $\Psi(z)=\frac{2\psi^*_{y_n}}{y_n}$.
Hence we obtain \eqref{HE}.
  	\end{proof}

\vskip10pt

\begin{lemma}\label{lemconti}
Let $\psi\in C^{1,1}(\overline{\mathbb R^n_+})$ be a solution to  \eqref{blow1} with $b>-1$.
Assume  $\psi_{x_n}(x', 0)=0\ \forall\ x'\in\R^{n-1}$,  and equation \eqref{blow1} is  uniformly elliptic.
Then $\psi\in C^{2,\alpha}(\overline{\mathbb R^n_+})$ for some $\alpha\in (0, 1)$.
%%% and the modulus of continuity of $D^2 \psi$ in $B^+_1(0) $ 
%%%  depends only on  $b, n$, $\|D_x^2 \psi\|_{L^\infty(\mathbb R^n_+)}$.  
\end{lemma}

\begin{proof}
Let $\psi^*$ be the partial Legendre transform of $\psi$.
Then $\psi^*$ satisfies equation \eqref{002} and  the assumptions of Lemma \ref{lemholder}. 
Hence by Lemma \ref{lemholder}, $\frac{\psi^*_{y_n}}{y_n}\in C^\alpha(\overline{\mathbb R^n_+})$. 
Recall that 
$\frac{\psi_{x_n}(x)}{x_n}=-\frac{\psi^*_{y_n}(y)}{y_n}$.
We therefore have
\begin{equation*}
\begin{split}
\Big|\frac{\psi_{x_n}(x)}{x_n}-\frac{\psi_{x_n}(\tilde x)}{\tilde x_n}\Big| 
        =\Big|\frac{\psi^*_{y_n}(y)}{y_n}-\frac{\psi^*_{y_n}(\tilde y)}{\tilde y_n}\Big|
\le C|y-\tilde y|^\alpha.
\end{split}
\end{equation*}
By the partial Legendre transform, $y_n=x_n,$ $y'= D_{x'}\psi$.
It follows that
\begin{equation*}
|y'-\tilde y'|=| D_{x'} \psi(x)- D_{x'}\psi(\tilde x)|\le \|D^2\psi\|_{L^\infty(\mathbb R^n_+)}|x-\tilde x|.\end{equation*}
Hence $\frac{\psi_{x_n}}{x_n}\in C^\alpha(\overline{\mathbb R^n_+})$ and we have the estimate
\begin{equation}\label{dptc}
\Big\|\frac{\psi_{x_n}}{x_n}\Big\|_{C^\alpha(\mathbb R^{n-1}\times[0,1])}\le C  
\end{equation}
for a constant $C$ depending only on  $b, n$, and $\|D_x^2 \psi\|_{L^\infty(\mathbb R^n_+)}$.

We make an even extension of $\psi(x)$ with respect to the variable $x_n$ and
still denote it by $\psi(x)$. 
Regard $\frac{\psi_{x_n}}{x_n}$ as a known function, which is H\"older continuous. 
Then we can write equation \eqref{blow1} in the form
\begin{equation*}
\mathcal F(x,  D^2 \psi)=1.
\end{equation*}
By our assumption,  $\mathcal F$ is  fully nonlinear, uniformly elliptic,
and is $C^\alpha$ smooth in $x$. Since $\mathcal F^{\frac 1n}$ is concave in $ D^2 \psi$,
by the $C^{2,\alpha}$ regularity \cite {Ca89}, we also conclude that $\psi\in C^{2,\alpha}(\R^n)$.
\end{proof}

\vskip10pt

In Lemma \ref{lemconti}, we assume that $\psi$ is $C^{1,1}$ and the equation \eqref{blow1} is uniformly elliptic.
Without these conditions, Lemma \ref{lemconti} does not hold. Here is an example.

\begin{example}
The function
$$u(x, y)=\frac12 x^2+\frac{|y|^{1+\eps}}{(1+\eps)\eps} . $$
is strictly convex and satisfies the equation
\beq\label{R3.1}
 (u_{xx}-1+|y|^{1-\eps}) u_{yy}-u_{xy}^2 =1. 
\eeq
But $u$ is not $C^{1,1}$ smooth.
\end{example}
 
With the aid of Lemma \ref{lemconti}, we can now prove Theorem \ref{thmbern}. 
 
\vskip10pt

\noindent
{\bf Proof of Theorem \ref{thmbern}:}\ 
Let $\psi$ be the solution in Theorem \ref{thmbern}.   Let
\begin{equation}
\psi^{m}(x)=\frac{\psi(mx)}{m^2},\quad m=1,2,\cdots
\end{equation}
be a blow-down sequence of $\psi$. 
Since \eqref{blow1} is uniformly elliptic for $\psi$, it is also uniformly elliptic for $\psi^m$ 
with the same ellipticity constants. 
The uniform ellipticity  implies that there is a constant $\hat C>0$, 
independent of $m$, such that  
\beq \label{cmc}
\hat C^{-1} \mathcal I\le \mathcal M_{\psi^m} \le \hat C\mathcal I, 
\eeq
where $\mathcal I$ is the unit matrix and $\mathcal M_{\psi}$ denotes matrix in equation \eqref{blow1}.
Hence the first entry in the matrix $\mathcal M_{\psi^m}$ satisfies 
$$\psi^m_{x_nx_n} +b \frac{\psi^m_{x_n}}{x_n}= \hat f , $$
for a function $\hat f$ satisfying $\hat C^{-1}  \le \hat f  \le \hat C$.
We can solve the above equation, regarding it as an ode with variable $x_n$,
\beq\label{psimb}
\psi^m(\cdot, x_n)=\psi^m(\cdot, 0)+\int_0^{x_n}  r^{-b}\int_0^r s^{b} \hat f(\cdot, s)ds.
\eeq
In \eqref{psimb} we have used the initial condition $\psi^m_{x_n} (\cdot, 0)=0$.
Note that \eqref{cmc} implies that $\psi^m(x', 0)=O(|x'|^2)$. Hence from \eqref{psimb} we have
$\psi^m(x)=O(|x|^2)$ near $0$. 

Hence by the assumptions in Theorem \ref{thmbern}, 
$\psi^{m}$ satisfies the conditions in Lemma \ref{lemconti}, uniformly in $m$.
Therefore,  by Lemma \ref{lemconti} we have
\begin{equation}
|D^2 \psi(x)-D^2 \psi(0)|=\lim_{m\rightarrow +\infty}\big|D^2 \psi^m\big( {\Small\text{$\frac xm$}} \big)-D^2 \psi^m(0)\big|=0 
\end{equation}
for any given point $x\in \R^n_+$.
That is, $D^2 \psi(x)=D^2 \psi(0)$ $\forall\ x\in \R^n_+$.
Hence $\psi$ is a quadratic polynomial.
By the assumption $\psi_{x_n}(x',0)=0\ \forall\ x'\in\R^{n-1}$, 
we have $c_{in}=0$ in the polynomial \eqref{qp1}. 
\hfill$\square$

 \vskip10pt

The following example shows that the Bernstein theorem \ref{thmbern} is not unconditionally true.

\begin{example} 
Let
$$\psi(x)={\Small\text{$\frac 12$}} (x_2^2+\cdots+x_{n-1}^2)
             +{\Small\text{$\frac 12$}} x_1^2x_n^{b-1}+{\Small\text{$\frac {x_n^{3-b}}{2(3-b)}$}} , $$
where $b>1$.
By direct computation, $\psi$ satisfies equation \eqref{blow1}.

%%% When $b=2$, $\psi$ is a smooth function. 
%%% When $b<1$, the function 
%%% $$\psi(x) = {\Small\text{$\frac 12$}} |x'|^2+{\Small\text{$\frac {x_n^2}{2(b+1)}$}} +{\Small\text{$\frac {x_n^{1-b}}{1-b}$}} $$
%%%  also satisfies equation \eqref{blow1}.
%%% {\Small\color{blue} (When $b=2$,  $\psi_{x_nx_n}+\frac {2}{x_n} \psi_{x_n} = \Delta_y \psi$ if $x_n=(y_1^2+y_2^2+y_3^2)^{1/2}$. 
%%%  Hence we do not expect a counterexample with $\psi_{x_n}=0$)}
\end{example}

It is well known that the classical Bernstein theorem for the Monge-Amp\`ere equation
was proved by J\"orgens for dimension $n=2$, Calabi for $3\le n\le 5$, and Pogorelov for all dimensions \cite{Po72}.
In \cite{JW13},  a Bernstein type theorem on a different singular Monge-Amp\`ere type equation in half space was proved,

\begin{example} 
[\cite{JW13}]
%%%  In \cite {JW13}, the authors proved a Bernstein theorem for the following singular Monge-Amp\`ere equation
 Let $u$ be a convex solution to 
\beq\label{MA31} 
{\begin{split}
\det D^2 u & =\big( {\Small\text{$\frac {u_{x_n}}{x_n} $}} \big)^{n+2}\ \ \ \text{in}\ \R^n_+, \\[-3pt]
u(x', 0) &={\Small\text{$\frac 12$}} |x'|^2.
\end{split}}
\eeq
Then either $u= \frac 12 |x|^2$, or $u(x)=\frac 12|x'|^2$.
\end{example}

In contrast to the above example, it is also interesting to mention the following  counter-example by Savin  \cite{S14}.

\begin{example}  [\cite{S14}]
Let 
$$
u(x) ={\Small\text{$ \frac {x_1^2} {2(1+x_n)}$}} +
         {\Small\text{$\frac 12$}} \big(x_2^2+\cdots +x_n^2\big)
          +{\Small\text{$\frac 16$}} x_n^3.
$$
Then $u$ satisfies 
\beq\label{MA41}
{\begin{split}
\det D^2 u & =1\ \ \ \text{in}\ \ \R^n_+, \\
u(x', 0) &={\Small\text{$\frac 12$}} |x'|^2.
\end{split}}
\eeq
\end{example}
 
%%% \begin{example} 
%%%  In \cite {DY}, the authors studied the following problem
%%% \beq\label{MA32}
%%%% \det D^2 u =e^{\frac 12 x\cdot Du -u} \ \ \text{in}\ \R^n.
%%% \eeq
%%% They proved that if $u$ is an entire convex solution to \eqref{MA32}, then $u$ is a quadratic function.
%%% \end{example}

\vskip10pt

\section{\bf $C^2$ regularity of free boundary}
%Section 3

In this section, we prove the $C^2$ regularity of the free boundary, which is equivalent to 
the $C^2$ regularity of  the tangent cone of  the following singular Monge-Amp\`ere equation  
\begin{equation}\label{MA-1}
\det D^2 u= {g(Du)} +\delta_0\quad \text{in}\quad B_1(0) .
\end{equation}
Assume that $g$ is a smooth function and satisfies
\beq \label{lgL}
0<\lambda \le g\le \Lambda<+\infty 
\eeq
for positive constants $\Lambda\ge \lambda>0$.
Let $u$ be a strictly convex solution to \eqref{MA-1}.
Assume that $u(0)=0$, $u\ge 0$ and $u>0$ on $\p B_1(0)$.

\begin{theorem}\label{T4.1}
Let $\phi$ be the tangential cone of $u$ at $0$.
Then  the section $S_{1,\phi}=\{x\in \R^n:\ \phi(x)<1\}$ is uniformly convex and $C^2$ smooth 
provided $g$ is $C^2$ smooth and satisfies \eqref{lgL}.
\end{theorem}

We have shown in \S2 that $S_{1,\phi}$ is uniformly convex and $C^{1,1}$ smooth.
In this section we raise the regularity of $S_{1,\phi}$  from $C^{1,1}$ to $C^2$.
This is the most delicate part in the proof of Theorem \ref{thmA}.
Our proof uses a blow-up argument, 
and uses the maximum principle in an infinite sequence of specially chosen domains.
%%% Auxiliary functions play a key role in the proof for the regularity of the Monge-Amp\`ere equation \cite{CNS,  S13}.

Denote 
\beq\label{zeta0}
\zeta=\frac ur,
\eeq
where $r=|x| \in (0, 1]$.
We can extend $\zeta$ continuously to $r=0$.
By \eqref{asyw}, $\zeta(\theta, 0)=\phi(\theta, r)/r$.
To prove Theorem \ref{T4.1}, we will prove that $\zeta(\theta, r) \in C^2(\mathbb S^{n-1}\times [0, 1])$,
namely  the second derivatives of $\zeta$ can continuously extend to $r=0$ (Theorem \ref{thm2}).
First we derive the equation for $\zeta$ in $(\theta, r)$.

\begin{lemma}\label{L4.1}
%%% Let $u$ satisfies  $\det D^2 u = g$ in $B_1(0)$. Let $\zeta=\frac ur$.
Let $\theta$ be an orthonormal frame on $\mathbb S^{n-1}$. 
Then  $\zeta$ satisfies 
 the Monge-Amp\`ere type equation
\begin{equation}\label{po1}
\det \begin{pmatrix}
\displaystyle\frac{\zeta_{rr}}{r^{n-2}}+\frac{2\zeta_r}{r^{n-1}} & \displaystyle\frac{\zeta_{r\theta_1}}{r^{\frac{n-2}{2}}} &\cdots& 
                    \displaystyle\frac{\zeta_{r\theta_{n-1}}}{r^{\frac{n-2}{2}}}\\ 
\displaystyle\frac{\zeta_{r\theta_1}}{r^{\frac{n-2}{2}}}& \zeta_{\theta_1\theta_1}+\zeta+r\zeta_r & \cdots&\zeta_{\theta_1\theta_{n-1}}\\ 
\cdots &\cdots &\cdots &\cdots  \\
 \displaystyle\frac{\zeta_{r\theta_{n-1}}}{r^{\frac{n-2}{2}}} &\zeta_{\theta_1\theta_{n-1}} & \cdots&\zeta_{\theta_{n-1}\theta_{n-1}}+\zeta+r\zeta_r
\end{pmatrix} =\bar g , 
\end{equation}
where  $\bar g=: g(Du)$ is a smooth function of $\zeta, r\zeta_r, \zeta_\theta$. 
The subscript $\theta$ means covariant derivatives on the sphere $\mathbb S^{n-1}$. 
\end{lemma}

\begin{proof}
For any given point $(\theta, r)\in \mathbb S^{n-1}\times (0, 1]$, there is no loss in assuming that  $\theta=0$ and $r=1$.
To derive equation \eqref{po1}, we use the spherical polar coordinates
%%% \begin{equation}\label{sphcor1}
%%%\begin{cases}
%%%x_1=r\cos{{\theta}}_{n-1}\cdots\cos{{\theta}}_1,\\
%%%x_2=r\cos{{\theta}}_{n-1}\cdots\cos{{\theta}}_2\sin{{\theta}}_1,\\
%%%\hskip30pt \cdots,\\
%%%x_{n-1}=r\cos{{\theta}}_{n-1}\sin{{\theta}}_{n-2},\\
%%%x_n=r\sin{{\theta}}_{n-1}
%%%\end{cases}
%%%\end{equation}
\begin{equation}\label{sphcor1}
\begin{cases}
x_1=r\sin{{\theta}}_1,\\
x_2=r\cos{{\theta}}_{1}\sin{{\theta}}_2,\\
\hskip30pt \cdots,\\
x_{n-1}=r\cos{{\theta}}_{1}\cdots\cos{{\theta}}_{n-2}\sin{{\theta}}_{n-1},\\
x_n=r\cos{{\theta}}_1\cdots\cos{{\theta}}_{n-1} .
\end{cases}
\end{equation}
In the local coordinates \eqref{sphcor1},  we have
\begin{equation*}
\begin{pmatrix}
 \frac{{\p} \theta_1}{{\p} x_1} & \cdots & \frac{{\p} \theta_{n-1}}{{\p} x_1}&\frac{{\p} r}{{\p} x_1}\\ 
 \cdots &  \cdots &  \cdots & \cdots\\
  \frac{{\p} \theta_1}{{\p} x_n} & \cdots & \frac{{\p} \theta_{n-1}}{{\p} x_n} &\frac{{\p} r}{{\p} x_n}
\end{pmatrix}
=\mathcal I\ \ \ \text{at}\  {{\theta}}=0 . 
\end{equation*}
By direct computation, we also have, at $\theta=0$, 
$$
 \frac{\p^2\theta_\alpha}{\p x_i\partial x_j} = \left\{
 {\begin{split}
  & -\delta_{\alpha j}\ \  \ \text{if}\ 1\le j\le n-1, i=n , \\
  & -\delta_{\alpha i}\ \  \ \text{if}\ 1\le i\le n-1, j=n , \\
  & 0\ \hskip30pt \text{otherwise}.
\end{split}} \right.
$$
 $$
 \frac{\p^2 r}{\p x_i\partial x_j} =  \bigg\{
{\begin{split} 
  & 1\ \ \ \ \  \text{if} \ i=j\le n-1, \\
  & 0\ \ \ \  \text{otherwise}.
\end{split} } $$ 
Rescaling back to $r\in (0, 1]$, we obtain
 \begin{equation}\label{diju}
D^2_{x_ix_j} u=\begin{cases}
{\small\text{$ \frac 1{r^2} $}} u_{{{\theta}}_i{{\theta}}_j}+ {\small\text{$\frac{u_r}{r} $}} \delta_{ij}
      \ \ \  \text{if}\ 1\le i,j\le n-1,\\[4pt]
{\small\text{$ \frac{u_{r{{\theta}}_i}}{r} $}} - {\small\text{$\frac{u_{{{\theta}}_i}}{r^2} $}} 
     \ \ \hskip27pt \text{if}\ 1\le i\le n-1, j=n,\\
u_{rr}  \quad \hskip55pt \text{if}\  i=j=n.
\end{cases}
\end{equation}
By the change $u=r\zeta$, we then obtain
$$D^2 u =
\begin{pmatrix}
r \zeta_{rr} +2\zeta_r & \zeta_{r\theta_1} &\cdots&   \zeta_{r\theta_{n-1}}\\[3pt]
 \zeta_{r\theta_1}  & \frac 1r \zeta_{\theta_1\theta_1}+\zeta_r+\frac 1r \zeta & \cdots& \frac 1r\zeta_{\theta_1\theta_{n-1}}\\[3pt] 
\cdots &\cdots &\cdots &\cdots  \\[3pt]
 \zeta_{r\theta_{n-1}} & \frac 1r\zeta_{\theta_1\theta_{n-1}} & \cdots&\frac 1r\zeta_{\theta_{n-1}\theta_{n-1}} +\zeta_r+\frac 1r \zeta
\end{pmatrix} .
$$
Hence we obtain \eqref{po1} at ${{\theta}}=0$.
As the  Monge-Amp\`ere equation is invariant under rotation of the coordinates, 
we see that \eqref{po1}  holds at a general point $(\theta, r)\in \mathbb S^{n-1}\times (0, 1)$.

To see that  $\bar g=: g(Du)$
is a smooth function of $\zeta, r\zeta_r, \zeta_{{\theta}}$, we just need to compute 
\beq\label{barg}
u_{x_i} 
 = (r\zeta)_{x_i}
   =(r\zeta)_{r}\frac{\p r}{\p x_i}+r\zeta_{{{\theta}}_j}\frac{\p {{\theta}}_j}{\p x_i}
  =\Big\{ {\begin{split}
   &  \zeta_{{{\theta}}_i} \ \hskip27pt \ i<n, \\
    & \zeta+r\zeta_r \ \ \ i=n,
    \end{split}}
    \ \ \ \text{at}\ \theta=0.
\eeq
 \end{proof}

Note that the local coordinates $\theta$ in \eqref{sphcor1} is not orthonormal on $\mathbb S^{n-1}$,
but it is second order close to an orthonormal frame at $\theta=0$. 
Hence we can use it to calculate the equation at $\theta=0$.
Here we say a coordinate system $\alpha=\alpha(\theta)$ is second order close to $\theta$ if
$\frac{d\alpha}{d\theta}=\mathcal I + O(|\theta|^2)$.
 By our notation,  $\zeta=\phi(\theta)+\frac{w}{r}$ (here we denote $\phi(\theta)=\phi(\theta, 1)$ for the function $\phi$ in \eqref{uwp}). Hence
\begin{equation}\label{zer}
{\begin{split}
   \zeta_{r} & =\frac{w_r}{r}-\frac{w}{r^2}, \\
  \zeta_{rr} & =\frac{w_{rr}}{r}-\frac{2w_r}{r^2}+\frac{2w}{r^3},\\
  \zeta_{r\theta} & =\frac{w_{r\theta}}{r}-\frac{w_\theta}{r^2},\\
   \zeta_{\theta\theta} &=\phi_{\theta\theta}+\frac{w_{\theta\theta}}{r}.
\end{split}}
\end{equation}
By Lemma \ref{L2.4},   all the entries in \eqref{po1} are uniformly bounded.  
To make \eqref{po1} uniformly elliptic, let 
$$s=r^{\frac n2}. \hskip40pt$$ 
Then
\begin{equation}\label{zes}
{\begin{split}
\zeta_s &=\frac{2}{n}r^{-\frac {n-2}{2}}\zeta_r,\\
\zeta_{s\theta}&=\frac{2}{n}r^{-\frac {n-2}{2}}\zeta_{r\theta},\\
\zeta_{ss} & =\frac{4\zeta_{rr}}{n^2 r^{n-2}}-\frac{2(n-2)\zeta_r}{n^2 r^{n-1}}.
\end{split}}
\end{equation}
Hence by  \eqref {zer},  \eqref{zes}, and Lemma \ref{L2.4}, we have

\begin{corollary}\label{C4.1}
As a function of $\theta$ and $s$, $\zeta$ satisfies $\zeta_{s}(\theta,0)=0$ and
$\zeta\in \! C^{1,1}(\mathbb S^{n-1}\!\times  [0, 1))$.
\end{corollary}

%%% In Corollary \ref{C4.1}, we have implicitly extended $\zeta$ to $s=0$. This is possible by the estimates in \S2.
By  \eqref {zes}, equation \eqref{po1} changes to
\begin{equation}\label{po2}
\det \begin{pmatrix}
 (\frac n2)^2 \zeta_{ss}+\frac{n(n+2)}{4}\frac{\zeta_s}{s}   &   \frac n2\zeta_{s\theta_{1}}&\cdots& \frac n2\zeta_{s\theta_{n-1}}\\[3pt]
  \frac n2\zeta_{s\theta_1} & \zeta_{\theta_1\theta_1}+\zeta+\frac n2 s\zeta_s & \cdots&\zeta_{\theta_1\theta_{n-1}}\\[3pt]
                    \cdots &\cdots &\cdots &\cdots  \\[3pt]
 \frac n2\zeta_{s\theta_{n-1}} &\zeta_{\theta_1\theta_ {n-1}} & \cdots&\zeta_{\theta_{n-1}\theta_{n-1}}+\zeta+\frac n2 s\zeta_s
\end{pmatrix}=\bar g .
\end{equation}
After the change $s=r^{\frac n2}$, we have $r\zeta_r=\frac n2 s\zeta_s$.
Hence by \eqref{barg}, $\bar g$ is still a smooth function, and $\bar g$  is analytic if $g$ is.

In the following we regard $\zeta$ as a function of $\theta$ and $s$, except otherwise specified.

\begin{lemma}\label{unif-ellip}
Equation \eqref{po2} is uniformly elliptic in  $\theta$ and $s$.
\end{lemma}

\begin{proof}
As shown in Corollary \ref{C4.1},  $\zeta$ is $C^{1,1}$ in $(\theta, s)$.
Note that all the entries of the matrix in \eqref{po2} are uniformly bounded.  
By our assumption, the function $\bar g>0$.
Hence the eigenvalues of the matrix are uniformly bounded and positive. 
Hence equation \eqref{po2} is uniformly elliptic.
\end{proof}

%%\subsection{A priori estimates in the spherical coordinates}

\begin{theorem}\label{thm2} 
Let $\zeta(\theta,s)\in C^{1,1}(\mathbb S^{n-1}\times [0,1))$ be 
a solution to \eqref{po2} and satisfy $\zeta_{s}(\theta,0)=0$.
Assume that $\bar g$ is positive and $C^2$ smooth.
Then $\zeta(\theta,s)\in C^{2}(\mathbb S^{n-1}\times[0,1))$.
\end{theorem}

Once Theorem \ref{thm2} is proved,  Theorem \ref{T4.1} follows by our definition of $\zeta$.

To prove Theorem \ref{thm2}, 
one may wish to apply the partial Legendre transform \eqref{plt} to \eqref{po2}.
But due to the term $s\zeta_s$, equation \eqref{po2} becomes very complicated  after the change \eqref{plt}.
In the following we will use a blow-up argument to prove Theorem \ref{thm2}.
By the $C^2$ regularity of $\zeta$ for $s>0$,
it suffices to prove  it at $\mathbb S^{n-1}\times\{s=0\}$. 
By Theorem \ref{thmbern}, a blow-up sequence converges to a quadratic polynomial $\psi$ of the form \eqref{qp1},
in which the mixed derivatives $\psi_{x_ix_n}$ vanish for $i<n$.
It implies that all the blow-up sequences sub-converge to the same limit 
and so the mixed derivatives $\zeta_{s\theta}$ are continuous at $s=0$.
For other second derivatives, 
we need to prove that all the blow-up sequences at a given point converge to the same quadratic polynomial,
namely the limit is independent of the blow-up sequences.
It implies not only the existence but also the continuity of the second derivatives.
For clarity we divide the proof into three lemmas.

For any given point $\theta_0\in\mathbb  S^{n-1}$, 
by a rotation of the coordinate system, we assume  $\theta_0=0$. 
By subtracting a linear function, there is no loss of generality in assuming that 
$\zeta(0)= D \zeta(0)=0$ in the following argument.

\begin{lemma}\label{lemconti1}
Let $\zeta$ be as in Theorem \ref{thm2}. Then $\zeta_{s\theta}\in C(\mathbb S^{n-1}\times[0,1))$.
\end{lemma}

\begin{proof}
We only need to prove that  
\begin{equation}\label{con-a1}
\lim_{s\rightarrow 0^+}\zeta_{s\theta}(\theta,s)=0. 
\end{equation}
If \eqref{con-a1} is not true, there exists a sub-sequence   $(\theta^k,s^k)\rightarrow (0,0)$ such that 
\begin{equation}\label{contra1}
\lim_{k\rightarrow +\infty}|\zeta_{s\theta}(\theta^k,s^k)|\ge \varepsilon_0
\end{equation}
for a constant $\varepsilon_0>0$. 
In this case, we make the coordinate transform 
\begin{equation}\label{coortran1}
{\begin{split}
\theta &=\lambda_k \varphi+\theta^k,\\
s &=\lambda_k \tau,
\end{split}}
\end{equation}
where $\lambda_k=s^k$, and set 
\begin{equation}\label{coortran2}
\zeta^k(\varphi,\tau)=\frac{\zeta(\theta,s)-\zeta(\theta^k,0)- D_\theta \zeta(\theta^k,0)\cdot (\theta-\theta^k)}{\lambda_k^2}.
\end{equation}
Note that  \eqref{coortran1} means a dilation and a rotation of the coordinates. We can regard $\varphi$ as an orthonormal frame on $\lambda_k^{-1}\mathbb S^{n-1}$. 
 Then by the $C^{1,1}$ regularity of $\zeta$, we have
\begin{equation*}
|\zeta^k(\varphi,\tau)|\le C(\tau^2+|\varphi|^2)
\end{equation*}
for a constant $C>0$ independent of $k$.
Moreover, $\zeta^k$ satisfies the equation
\begin{equation}\label{po3}
\det \begin{pmatrix}
(\frac n2)^2 \zeta^k_{\tau\tau}+\frac{n(n+2)}{4}\frac{\zeta^k_\tau}{\tau} 
                  &   \frac n2\zeta^k_{\tau\varphi_{1}}&\cdots& \frac n2\zeta^k_{\tau\varphi_{n-1}}\\[5pt]
 \frac n2\zeta^k_{\tau\varphi_1}& \zeta^k_{\varphi_1\varphi_1}+h^k(\varphi,\tau) 
                  & \cdots&\zeta^k_{\varphi_1\varphi_{n-1}}\\[3pt] 
    \cdots &\cdots &\cdots &\cdots  \\[3pt]
\frac n2\zeta^k_{\tau\varphi_{n-1}} &\zeta^k_{\varphi_1\varphi_{n-1}} & \cdots&\zeta^k_{\varphi_{n-1}\varphi_{n-1}}+h^k(\varphi,\tau)
\end{pmatrix}=\bar g^k
\end{equation}
where 
$h^k=\lambda_k^2( \zeta^k+\frac n2 \tau\zeta_\tau)+\zeta(\theta^k,0)+\lambda_k D_{\theta} \zeta(\theta^k,0)\cdot \varphi$, 
and $\bar g^k$ is a smooth function converging to a positive constant.

Write \eqref{po3} as a general fully nonlinear elliptic equation of the form
\begin{equation}\label{Fk}
\mathcal F_k(\varphi,\tau,\zeta^k,D\zeta^k,D^2\zeta^k)=0. 
\end{equation}
By Lemma \ref{unif-ellip}, $\mathcal F_k$ is  uniformly elliptic.
Moreover, $\mathcal F_k^{\frac 1n}$ is concave with respect to its variables $D^2 \zeta^k$ 
and is $C^{1,1}$ smooth in all arguments in $\tau>0$.
Hence by Evans-Krylov's interior regularity theory \cite{GT},  we have
\begin{equation}\label{con-a2}
\|\zeta^{k}\|_{C^{3,\alpha}(\Omega)}\le C_{\Omega}\quad \forall\ \Omega\subset\subset \mathbb R^n_+ ,
\end{equation}
where the constant $C_\Omega$ is independent of $k$. 
By passing to a subsequence, we have
\begin{equation}
\zeta^k(\varphi,\tau)\rightarrow \bar \zeta(\varphi,\tau)
  \quad \text{in}\quad 
   C^{3,\alpha}_{loc}(\mathbb R^n_+)\cap C^{1,1-\eps}_{loc}(\overline{\mathbb R^n_+}) 
\end{equation}
for a function 
$\bar \zeta \in C^{3,\alpha}_{loc}(\mathbb R^n_+)\cap C^{1,1}_{loc}(\overline{\mathbb R^n_+})$,
where $\eps>0$ is any small constant, and $\overline{\mathbb R^n_+} = {\mathbb R^n_+}\cup \{x_n=0\}$.
Moreover, the $C^{1,1}$ norm of $\bar\zeta$ is independent of the choice of the blow-up sequences,
it depends only on the $C^{1,1}$ norm of $\zeta$. 
Hence $\bar \zeta$ satisfies equation \eqref{blow1} with $b=\frac{n+2}{n}$ under the coordinates $x'=\varphi$, $x_n=\tau$.  
By Theorem \ref{thmbern}, $\bar \zeta$ is a quadratic polynomial of the form
\begin{equation} \label{bar-zeta}
\bar \zeta(\varphi,\tau)=\frac{1}{2}c_{00}\tau^2+\frac{1}{2} {\Small\text{$  \sum_{i,j=1}^{n-1} $}} c_{ij}\varphi_i\varphi_j.
\end{equation} 
Note that in \eqref{bar-zeta}, the mixed derivatives $\bar \xi_{\tau\varphi}(0',1)=0$.
By the interior regularity for \eqref{Fk}, it implies that 
\begin{equation}\label{zetast}
\lim_{k\rightarrow +\infty}\zeta_{s\theta}(\theta^k,s^k)
   =\lim_{k\rightarrow +\infty}\zeta^k_{\tau\varphi}(0',1)=\bar \xi_{\tau\varphi}(0',1)=0.
\end{equation}
By the interior regularity \eqref{con-a2}, we see that \eqref{zetast}
is in contradiction with \eqref{contra1}.  This completes the proof.
\end{proof}

\begin{lemma}\label{lemconti2}
Let $\zeta$ be as in Theorem \ref{thm2}. Then $\zeta_{ss}\in C( \mathbb S^{n-1}\times[0,1))$.
\end{lemma}

\begin{proof}
To prove the continuity of $\zeta_{ss}$ on $\mathbb S^{n-1}\times\{s=0\}$,  we need to prove the limit
$\lim_{s\rightarrow 0^+, \theta\to 0}\zeta_{s s}(\theta,s)$ exists.
By Lemma \ref{unif-ellip}, $ \zeta_{s s}(\theta, s)$ is uniformly bounded.
Hence there is a sub-sequence  $s^k\rightarrow 0$ such that 
$\zeta_{ss}(0, s^k)$ is convergent.
We introduce the coordinates $(\varphi,\tau)$ and function $\zeta^k$ 
as in \eqref{coortran1} and \eqref{coortran2}, with $\theta^k=0$.

By the proof of Lemma \ref{lemconti1}, we have
\begin{equation} \label{zeta12}
\zeta^k(\varphi,\tau)\rightarrow \frac{1}{2}c_{00}\tau^2+\frac{1}{2} {\Small\text{$  \sum_{i,j=1}^{n-1} $}} c_{ij}\varphi_i\varphi_j
  \quad \text{in}\quad  C^{3,\alpha}_{loc}(\mathbb R^n_+)\cap C^{1,1-\eps}_{loc}(\overline{\mathbb R^n_+}),
\end{equation}
by passing to a subsequence if necessary.
To prove Lemma \ref{lemconti2}, we need to prove that 
\beq\label{zeta00}
\lim_{s\rightarrow 0^+, \theta\to 0}\zeta_{s s}(\theta,s)=c_{00}.
\eeq

By the convergence \eqref{zeta12} and the interior regularity of equation \eqref{po3}, 
we can choose a subsequence, such that
\begin{equation}\label{zeta-c00}
\Big\|\frac{\zeta^k_\tau(\varphi,\tau)}{\tau}-c_{00}\Big\|_{L^\infty(Q_k)}\le \frac 1{2k}
      \quad \text{in}\ Q_k=:\big\{(\varphi,\tau):\  |\varphi|\le 1,\	\frac 1k\le \tau\le 1\big\}. 
\end{equation}
Scaling back to $\zeta(\theta,s)$, we obtain
\begin{equation}\label{aym-aa2}
\Big\|\frac{\zeta_s(\theta,s)}{s}-c_{00}\Big\|_{L^\infty(\Sigma_k)}\le \frac 1{2k}
 \quad\ \text{in}\  \Sigma_k=:\big\{(\theta,s):\ |\theta|\le \lambda_k,\	\frac {\lambda_k}k\le s\le \lambda_k\big\}. 
\end{equation}
Let $t=\frac{s^2}{4}$.  The above estimate implies that
\begin{equation}\label{aym-a2}
\left\|\zeta_t-2c_{00}\right\|_{L^\infty(\widehat\Sigma_k)}\le \frac 1k
\quad \text{in}\ \widehat\Sigma_k=\big\{(\theta,t):\ |\theta|\le \lambda_k,\	\frac {\lambda^2_k}{4k^2}\le t\le \frac{\lambda_k^2}{4}\big\}. 
\end{equation}
Moreover, equation \eqref{po2} changes to 
\begin{equation}\label{max2}
\det \begin{pmatrix}
  t\zeta_{tt}+\frac{n+1}{n}\zeta_t&    \zeta_{t\theta_{1}}&\cdots&    \zeta_{t\theta_{n-1}}\\[5pt] 
    t\zeta_{t\theta_1}& \zeta_{\theta_1\theta_1}+\zeta+ t\zeta_t & \cdots&\zeta_{\theta_1\theta_{n-1}}\\[3pt] 
      \cdots &\cdots &\cdots &\cdots   \\[3pt]
   t\zeta_{t\theta_{n-1}} &\zeta_{\theta_1\theta_{n-1}} & \cdots&\zeta_{\theta_{n-1}\theta_{n-1}}+\zeta+ t\zeta_t
\end{pmatrix}=\frac{4\bar g}{n^2}.
\end{equation}
Note that in the matrix \eqref{max2},  there is a coefficient $t$ in the first column.
This is due to $\zeta_{s\theta} = t^{1/2} \zeta_{t\theta}$.
After the change $t=\frac{s^2}{4}$, we have $s\zeta_s=2 t\zeta_t$.
Hence by \eqref{barg}, $\bar g$ is still a positive and smooth function of its arguments.

Writing  equation \eqref{max2} as $\log \det W_{ij} =\log \frac{4\bar g}{n^2}$ and differentiating in $t$, we have
$$W^{ij} \p_t W_{ij} = \p_t \log \frac{4\bar g}{n^2}, $$
where $\{W^{ij}\}$ is the inverse of $\{W_{ij}\}$.
Dividing the above equation by $W^{11}$, and denoting $V=\zeta_t$, 
we can write the above equation  as
\begin{equation}\label{max3}
\begin{split}
\mathcal L(V)=:\ &tV_{tt}+  {\small\text{$  \frac{2n+1}{n} $}} V_{t}+  {\Small\text{$ \sum_{i,j=1}^{n-1} $}} a^{ij}V_{\theta_i\theta_j}\\
 & +  {\Small\text{$ \sum_{i,j=1}^{n-1}$}} t\zeta_{t\theta_j}  \tilde a^{ij} V_{t\theta_i}
 +  {\Small\text{$ \sum_{i,j=1}^{n-1} $}}  \left(\zeta_{t\theta_i}b^{ij}+ b^j\right)V_{\theta_j}=h
\end{split}
\end{equation}
where $a^{ij}, \tilde a^{ij},b^{ij} $, $b^j$ and $h$ are  continuous functions of the elements in the matrix in \eqref{max2},
namely $t,\zeta, \zeta_t, \zeta_{\theta_k\theta_l}, t\zeta_{tt},t\zeta_{t\theta_k}\zeta_{t\theta_l}$. 
From the assumptions in Theorem \ref{thm2}, all the elements are uniformly bounded, namely   
\begin{equation}\label{aym-a1}
|\zeta_t|+|t\zeta_{tt}|+|t\zeta_{t\theta_i}\zeta_{t\theta_j}|+| D^2_{\theta}\zeta|\le C
\quad \forall\ (\theta,t)\in \mathbb S^{n-1}\times [0,   {\Small\text{$\frac 1{16}$}} ] .
\end{equation}
It implies that $a^{ij}, \tilde a^{ij},b^{ij}, b^j,h$ are uniformly bounded.

Let $\eta$ be a cut-off function of $\theta$ such that
\begin{equation*}
\eta(\theta)\equiv 1\ \ \text{when} \  |\theta|\le  {\Small\text{$ \frac{1}{2} $}} ,\quad 
\eta \equiv 0\ \ \text{when}\ |\theta|>1,\ \ \ 
\text{and}  \ 0\le \eta\le 1.
\end{equation*}
Denote $\eta_k(\theta)=\eta\big(\frac{\theta}{\lambda_k}\big)$ and  $V_k=\eta_k V$. 
Then   $V_k$ satisfies 
\begin{equation*}
\mathcal L(V_k)=\eta_k h-[\eta_k,\mathcal L] V=: h_k,
\end{equation*}
where 
$ [\eta_k,\mathcal L]  V =\eta_k\mathcal LV-\mathcal L(\eta_k V) $.  
By \eqref{aym-a1}, we have 
\begin{equation}\label{hkc}
|h_k|\le C(1+\lambda_k^{-2}+\lambda_k^{-1}t^{-\frac 12})\le C\lambda_k^{-1}t^{-\frac 12}
          \ \ \  \text{when}\ \ 0<t<\lambda_k^2 ,
\end{equation}
 where $C$ is a positive constant independent of $k$.

Denote  $\delta_{k, -1}  = 1$,  $ \delta_{k,0}  =\frac{1}{4k^2}$. Let
\begin{equation}\label{dekl}
\Big\{ {\begin{split}
 & \delta_{k,l+1} = (\delta_{k,l})^{1+\frac n4},\\ 
 & \varepsilon_{k,l} =C_1 \delta_{k,l}^{\frac12+\frac 1n}\lambda_k^{\frac 2n},
  \end{split}} \ \ \ \quad l=0,1,2,\cdots ,
\end{equation}
where $C_1$ is a fixed large constant.  

{\it Claim}: For any given $m\ge 1$, we have
\begin{equation}\label{claim-conti1}
|V_k-2\eta_k c_{00}|\le  {\small\text{$ \frac 1k $}} +C_1  {\Small\text{$ \sum_{l=0}^{m-1} $}}    \delta_{k,l}^{\frac 14}
     \ \ \text{when}\ |\theta|\le \lambda_k,\  \delta_{k,m}\lambda_k^2\le t\le \delta_{k,m-1}\lambda_k^2.
\end{equation}

We  prove \eqref{claim-conti1} by induction. 
First note that by  \eqref{aym-a2},  
$$
|V_k-2\eta_k c_{00}|\le  {\small\text{$ \frac 1k  $}}
     \ \ \text{when}\ |\theta|\le \lambda_k,\  \delta_{k,m}\lambda_k^2\le t\le \delta_{k,m-1}\lambda_k^2,
$$
which is exactly \eqref{claim-conti1} for  $m=0$.
Assuming that \eqref{claim-conti1} holds for $m$, we prove that it holds for $m+1$. 
We introduce the auxiliary functions
\begin{equation}\label{sigmapm}
\sigma_{k,m}^{\pm}(\theta,t)
  =2\eta_k c_{00}\pm\Big(\frac 1k+ C_1{\Small\text{$ \sum_{l=0}^{m-1} $}} \delta_{k,l}^{\frac 14}+\varepsilon_{k,m} t^{-\frac 1n}\Big).
\end{equation}
By our choice of $\eta_k$ and \eqref{hkc},  we have
\begin{equation}
\begin{split}
\mathcal L(V_k-\sigma_{k,m}^{+})
  = & h_k-2c_{00}\mathcal L(\eta_k)+\frac{\varepsilon_{k,m}}{n}t^{-1-\frac 1n}\\
 \ge & \frac{t^{-\frac 12}}{n}(\varepsilon_{k,m} t^{-\frac 12-\frac 1n}-C\lambda_k^{-1}) \\
   > &0 \ \ \ \ \ 
     \quad \text{when}\ 0< t\le \delta_{k,m}\lambda_k^2
\end{split} 
\end{equation}
if the constant $C_1$ in \eqref{dekl} is chosen large.
By our induction assumptions, we have
\begin{equation*}
\begin{split}
     &V_k-\sigma_{k,m}^{+}\le 0\quad \text{if}\  |\theta|\le \lambda_k,\  t=\delta_{k,m}\lambda_k^2,\\
     &V_k-\sigma_{k,m}^{+}=-\Big(\frac 1k+ C_1{\Small\text{$  \sum_{l=0}^{m-1} $}}  \delta_{k,l}^{\frac 14}+\varepsilon_{k,m} t^{-\frac 1n}\Big)<0
           \quad \text{if}\  |\theta|=\lambda_k,\\
\ \hskip15pt \ & \lim\sup_{t\rightarrow 0^+}(V_k-\sigma_{k,m}^{+})<0.
\end{split}
\end{equation*}
By the maximum principle, it follows that 
\begin{equation*}
      V_k-\sigma_{k,m}^+\le 0
  \quad \text{if}\  |\theta|\le \lambda_k,\    0< t\le \delta_{k,m}\lambda_k^2.
\end{equation*}
Similarly, we have
\begin{equation*}
    V_k-\sigma_{k,m}^-\ge  0
   \quad \text{if}\  |\theta|\le \lambda_k,\  0< t\le \delta_{k,m}\lambda_k^2.
\end{equation*}
Therefore we obtain
\begin{equation*}
\begin{split}
|V_k-2c_{00}\eta_k|
     \le &\frac 1k+C_1 {\Small\text{$ \sum_{l=0}^{m-1} $}} \delta_{k,l}^{\frac 14}+\varepsilon_{k,m}t^{-\frac 1n}\\
     \le & \frac 1k+C_1 {\Small\text{$ \sum_{l=0}^{m-1} $}} \delta_{k,l}^{\frac 14}+\varepsilon_{k,m}\delta_{k,m+1}^{-\frac 1n}\lambda_k^{-\frac 2n}\\
     \le & \frac 1k+C_1{\Small\text{$ \sum_{l=0}^{m} $}} \delta_{k,l}^{\frac 14}
          \quad \text{if}\ \  |\theta|\le \lambda_k,\ \delta_{k,m+1}\lambda_k^2\le t\le \delta_{k,m} \lambda_k^2 .
\end{split}
\end{equation*}
The claim \eqref{claim-conti1} is proved.

For any  point $(\hat \theta,\hat t)$ near $(0, 0)$ with $\hat t>0$,
we can choose $k>0$ such that 
$$(\hat \theta,\hat t)\in \big\{(\theta,t):\ |\theta|\le\frac{\lambda_k}{2}, 0<t\le \frac{\lambda_k^2}{4}\big\} .$$
We then choose $m\ge 0$ such that $\delta_{k,m+1}\lambda_k^2\le \hat t\le \delta_{k,m}\lambda_k^2$. 
Hence we have
\begin{equation}\label{est1}
{\begin{split}
|V_k-2\eta_k c_{00}|
  & \le \frac 1k+C_1 {\Small\text{$ \sum_{l=0}^{m} $}} \delta_{k,l}^{\frac 14}\\
  &  \le  \frac 1k+C_1 {\Small\text{$ \sum_{l=0}^{\infty} $}} \Big(\frac 1{4k^2}\Big)^{{(1+n/4)^l}/{4}}   \\
  & \le \frac {C_1}{\sqrt k}\ \ \ \text{at}\ (\hat \theta,\hat t).
  \end{split}}
\end{equation}

Because $(\hat \theta,\hat t)$ is an arbitrary point near $(0, 0)$ with $\hat t>0$. Hence from \eqref{est1}
we conclude that (recall that $V=\zeta_t=2\frac{\zeta_s(\theta, s)}{s}$)
\begin{equation}\label{zss}
\lim_{\theta\to 0,  s\to 0^+}\frac{\zeta_s(\theta, s)}{s}=\frac 12\lim_{\theta\to 0,  s\to 0^+}V(\theta, s)=c_{00}. 
\end{equation}

The convergence \eqref{zss} implies that the constant $c_{00}$ in the blow-up limit \eqref{bar-zeta}
is independent of the choice of the blow-up sequence. 
Hence by the blow-up argument in the proof of Lemma \ref{lemconti1}, we infer that
\begin{equation}\label{ss00}
\lim_{\theta\to 0,  s\to 0^+}\zeta_{ss}(\theta, s)=c_{00}.
\end{equation}
By the convergence \eqref{ss00},  
we can define $\zeta_{ss}$ on $\mathbb S^{n-1}\times\{s=0\}$ 
as the limit $\lim_{ s\to 0^+}\zeta_{ss}(\theta, s)$.
The above proof also implies that $\zeta_{ss}$ is continuous on $\{s=0\}$.
For if not, let us assume that $\zeta_{ss}$ is dis-continuous at $(\theta, s)=(0,0)$.
Then there exist two sequences $(\theta^k_1, s^k_1)\to 0$ and $(\theta^k_2, s^k_2)\to 0$ 
such that $\zeta_{ss}(\theta^k_1, s^k_1)$ and $\zeta_{ss}(\theta^k_2, s^k_2)$ 
converge to different limits,
which is in contradiction with \eqref{ss00}.
This completes the proof.
\end{proof}

\vskip5pt

 \begin{remark}
In the first paragraph of the above proof, 
we must choose the sequence $(\theta^k, s^k)$ on the line $\{\theta=0\}$,
namely we must assume $\theta^k=0$.
Otherwise if $\frac{|\theta^k|}{s^k}$ is too large,
by the changes \eqref{coortran1}, \eqref{coortran2}, and the definition of the cube $Q_k$,
the set $\Sigma_k$ and $\widehat\Sigma_k$ may not intersect with the line $\{\theta=0\}$.
In this case, we can only obtain the convergence $\zeta_{ss}(\theta, s)$ for $\theta, s$ near $(\theta^k, s^k)$,
but not the full convergence  \eqref{zss} and \eqref{ss00}.
 \end{remark}
  
\vskip5pt

\begin{lemma}\label{lemconti3}
Let $\zeta$ be as in Theorem \ref{thm2}. Then $ D_{\theta}^2\zeta\in C(\mathbb S^{n-1}\times[0,1))$.
\end{lemma}

\begin{proof} 
To prove the continuity of $  D_{\theta}^2\zeta$ on $\mathbb S^{n-1}\times\{s=0\}$,  we need to prove the limit
$\lim_{s\rightarrow 0^+, \theta\to 0}   D_{\theta}^2\zeta (\theta,s)$ exists.
By Lemma \ref{unif-ellip}, $   D_{\theta}^2 \zeta (\theta, s)$ is uniformly bounded.
Hence there is a sub-sequence  $s^k\rightarrow 0^+$ such that 
$  D_{\theta}^2 \zeta (0, s^k)$ is convergent.
We introduce the coordinates $(\varphi,\tau)$ and function $\zeta^k$ 
as in \eqref{coortran1} and \eqref{coortran2}, with $\theta^k=0$.
  
%%% By the proof of Lemma \ref{lemconti1}, we have
%%% \begin{equation}\label{zeta13}
%%% \zeta^k(\varphi,\tau)\rightarrow \frac{1}{2}c_{00}\tau^2+\frac{1}{2}\sum_{i,j=1}^{n-1}c_{ij}\varphi_i\varphi_j
%%%%  \quad \text{in}\quad C^{3,\alpha}_{loc}(\mathbb R^n_+)\cap C^{1,1-\eps}_{loc}(\overline{\mathbb R^n_+}).
%%% \end{equation}

Let $\mathbb B$ denote the set of all convergent blow up sequences $\{\zeta^k\}$ given by \eqref{coortran2} (with $\theta^k\equiv 0$).
For any fixed unit vector $\gamma\in \mathbb R^{n-1}$, define
\begin{equation}\label{blow-inf}
c_{\gamma\gamma}=\inf_{\{\zeta^k\} \in \mathbb B}\lim_{k\rightarrow +\infty} \zeta^k_{\gamma\gamma}(0',1).
\end{equation} 
where $\zeta^k_{\gamma\gamma}=\zeta^k_{\theta_i\theta_j}\gamma_i\gamma_j$.
By a diagonal process, we can extract a subsequence in $\mathbb B$,
which for simplicity we still denote as  $\{\zeta^k\} $, such that
\begin{equation}
c_{\gamma\gamma}=\lim_{k\rightarrow +\infty}  \zeta^k_{\gamma\gamma}(0',1).
\end{equation}\label{cgg}
We {\it claim}
\begin{equation}\label{add2}
 \varlimsup_{\theta\to 0, s\to 0^+} \zeta_{\gamma\gamma}(\theta, s) \le c_{\gamma\gamma}.
\end{equation}

Indeed,
by the convergence \eqref{zeta12} and the interior regularity of equation \eqref{po3},
similarly to \eqref{zeta-c00} we can pass to a subsequence such that
\begin{equation*}
\big\| \zeta^k_{\gamma\gamma}(\varphi,\tau)-c_{\gamma\gamma}\big\|_{L^\infty(Q_k)}\le \frac 1k\ \   \text{in}\  Q_k.
\end{equation*}
Scaling back to $\zeta(\theta,s)$, this implies 
\begin{equation}\label{sigmak}
\big\|\zeta_{\gamma\gamma}(\theta,s)-c_{\gamma\gamma}\big\|_{L^\infty(\Sigma_k)}\le \frac 1k \ 
 \ \text{in}\ \Sigma_k,
\end{equation}
Here the domains $Q_k, \Sigma_k$  are the same as in
\eqref{zeta-c00} and \eqref{aym-aa2}.
%Moreover, $\zeta$ satisfies equation \eqref{max2}.

To simplify the notation, let us denote the matrix in \eqref{max2} as $R=(r_{ij})_{i,j=1}^n$, 
and rewrite equation \eqref{po2} as 
\begin{equation}\label{po2-v1}
\mathcal F(r_{ij})=: \log (\det R)=\log \bar g. 
\end{equation}
Then  $\mathcal F$ is concave in its variables $r_{ij}$. 
Differentiating \eqref{po2-v1} in direction $\gamma$ twice and by the concavity, we have 
\begin{equation*}
\mathcal F_{r_{ij}}r_{ij,\gamma\gamma}\ge (\log \bar g)_{\gamma\gamma}.
\end{equation*}
Denote $V=\zeta_{\gamma\gamma}$.
Similarly to \eqref{max3}, one obtains
\begin{equation}\label{linear-1}
\begin{split}
\mathcal L(V)=: 
V_{ss}+\frac{n+2}{n}\frac{V_{s}}{s}+ {\Small\text{$ \sum_{i,j=1}^{n-1} $}} a^{ij}V_{\theta_i\theta_j}
  + {\Small\text{$ \sum_{i=1}^{n-1} $}} a^{i}V_{\theta_is}
+ {\Small\text{$ \sum_{i=1}^{n-1} $}} b^i V_i+b^0 V_s    \ge h ,
\end{split}
\end{equation}
where $a^{ij},a^{i},b^i,b^0,h$ are continuous functions of
$
\theta, s,\zeta, D \zeta, D^{2} \zeta,\frac{\zeta_s}{s}.
$
Here $D \zeta$ and $ D^{2}\zeta$ denotes derivatives of $\zeta$ with respect to both $s$ and $\theta$.
Note that the operator in \eqref{linear-1} is different from that in \eqref{max3} but we use the same notation $\mathcal L$.
When differentiating \eqref{po2-v1} in $\gamma$ twice to obtain \eqref{linear-1},
we need to exchange the derivatives in $\gamma$ and $\theta$.
By the Ricci identity,  it arises some second derivatives of $\zeta$. 
They are all merged to the RHS $h$.

We now introduce the auxiliary function $\sigma_{k,m}^{\pm}$ as in \eqref{sigmapm},
and apply the same argument thereafter to obtain \eqref{add2}.

\begin{remark}
 In the proof of Lemma \ref{lemconti2},  we differentiate equation \eqref{max2} in $t$ once
to obtain equation \eqref{max3}.  Here
we differentiate equation \eqref{max2} in $\theta$ twice and use the concavity of $\mathcal F$ 
to obtain the inequality \eqref{linear-1}. 
Hence by the auxiliary function $\sigma^+$, 
we obtain  the inequality \eqref{add2}.
We cannot obtain the reverse of the inequality by the auxiliary function $\sigma^-$.
 \end{remark}

\begin{remark} 
Combining the definition of $c_{\gamma\gamma}$ in \eqref{blow-inf} and the estimate \eqref{add2}, 
we  obtain
\begin{equation}\label{add2+}
 \lim_{s\to 0^+} \zeta_{\gamma\gamma}(0, s) = c_{\gamma\gamma}.
\end{equation}
Note that the convergence in \eqref{add2+} is on the line $\theta=0$.
 \end{remark}

To prove the convergence $ \lim_{\theta\to 0, s\to 0^+} \zeta_{\gamma\gamma}(\theta, s)=c_{\gamma\gamma}$, 
we make use of the equation \eqref{po2}.
By Lemmas \ref{lemconti1} and \ref{lemconti2}, and noting that $ s \zeta_s = o(1)$ near $s=0$, 
we can write \eqref{po2} as 
\begin{equation}\label{add1}
\det( D^2_{\theta} \zeta+\zeta I) 
=\frac{2\bar g(0)}{n(n+1)c_{00}}+o(1) \ \text{ for }  (\theta,s)\text{ near }(0',0).
\end{equation}
The left hand side is the standard Monge-Amp\`ere operator on the sphere $ \mathbb S^{n-1}$.
For a $k\times k$ positive definite matrix $W$,
its determinant $\det W$ can be written as \cite{FO11}
\begin{equation}\label{Ha1}
\det W=\min_{(\nu_1,\cdots,\nu_k)\in SO(k)} {\scriptsize\text{$\prod$}}_{i=1}^k \nu_i^T W \nu_i
\end{equation} 
where $SO(k)$ is the set of orthogonal matrices and $\nu_i$ are column vectors of the matrix.

Consider a blow up sequence $\{ \zeta^k\}\in\mathbb B$. 
By the convergence \eqref{zeta12} and the interior regularity we may assume that, after passing to a subsequence,
\begin{equation}\label{zta}
\lim_{k\to \infty, (\theta,s)\in \Sigma_k}( D_\theta^2 \zeta+\zeta I)\rightarrow A
\end{equation}
for a positive definite matrix $A$, where $\Sigma_k$ is given in \eqref{aym-aa2}.
%%% Note that by the definition of $c_{\gamma\gamma}$ and 
%%% $\varlimsup_{\theta\to 0, s\to 0^+} \zeta_{\gamma\gamma}(\theta, s) \le c_{\gamma\gamma}$, we infer that 
%%% \begin{equation}\label{add4}
%%% \lim_{s\to 0^+} \zeta_{\gamma\gamma}(0, s) = c_{\gamma\gamma}.
%%%  \end{equation}
 %%%  From (4.42) and (4.43), can we conclude that 
%%%  \beq \lim_{\theta\to 0, s\to 0^+} \zeta_{\gamma\gamma}(\theta, s)=c_{\gamma\gamma} \ ?  \eeq

 Let $O$ be an orthogonal matrix such that $O^T A O$ is diagonal. Then we have
\begin{equation}
\begin{split}
\det ( D_\theta^2 \zeta+\zeta I)
   &=\det  (O^T( D_\theta^2 \zeta+\zeta I)O) \\
   &= {\scriptsize\text{$\prod$}}_{i=1}^{n-1} (\zeta_{\gamma^{(i)}\gamma^{(i)}}+\zeta)+o(1)
   \quad \text{in}\ \Sigma_k
\end{split}
\end{equation}
where $\{\gamma^{(1)}, \cdots, \gamma^{(n-1)}\}$ is an orthonormal basis of $\mathbb R^{n-1}$. 
By \eqref{add2} it then follows 
\begin{equation*}
{\begin{split}
\det A 
 & = \lim_{\Sigma_k\ni (\theta,s)\rightarrow (0',0)} {\scriptsize\text{$\prod$}}_{i=1}^{n-1} (\zeta_{\gamma^{(i)}\gamma^{(i)}}+\zeta) \\
 & \le  {\scriptsize\text{$\prod$}}_{i=1}^{n-1}(c_{\gamma^{(i)}\gamma^{(i)}}+\zeta(0)).
 \end{split}}
\end{equation*}
By the definition of $c_{\gamma\gamma}$ in \eqref{blow-inf} and the convergence \eqref{zta}, we have
\begin{equation*}
 {\scriptsize\text{$\prod$}}_{i=1}^{n-1}(c_{\gamma^{(i)}\gamma^{(i)}}+\zeta(0))
  \le \lim_{\Sigma_k\ni (\theta,s)\rightarrow (0',0)} {\scriptsize\text{$\prod$}}_{i=1}^{n-1} (\zeta_{\gamma^{(i)}\gamma^{(i)}}+\zeta)
 = \det A.
\end{equation*}
Combining the above two inequalities, we obtain
\begin{equation}\label{add3}
 {\scriptsize\text{$\prod$}}_{i=1}^{n-1}(c_{\gamma^{(i)}\gamma^{(i)}}+\zeta(0))
    =\det A=\frac{2\bar g(0)}{n(n+1)c_{00}}.
\end{equation}  

We are ready to conclude the continuity of $ D_\theta^2 \zeta$.
Indeed, by \eqref{add1} and \eqref{Ha1}, we have
\begin{equation}\label{add4}
\begin{split}
   \frac{2\bar g(0)}{n(n+1)c_{00}} 
   &=  \lim_{\theta\to 0, s\to 0^+}\det ( D^2_\theta \zeta+\zeta I) \\
   & \le  \varlimsup_{\theta\to 0, s\to 0^+} {\scriptsize\text{$\prod$}}_{i=1}^{n-1} (\zeta_{\gamma^{(i)}\gamma^{(i)}}+\zeta) .
   \end{split}
\end{equation}
By \eqref{add2} and \eqref{add3}, the RHS of \eqref{add4}
\begin{equation}\label{add5}
 \hskip70pt  \le  {\scriptsize\text{$\prod$}}_{i=1}^{n-1} (c_{\gamma^{(i)}\gamma^{(i)}}+\zeta(0))=\det A.
\end{equation}
Therefore the inequalities in \eqref{add4}  and \eqref{add5} becomes equalities. 
Choosing the coordinates such that $\gamma^{(1)}, \cdots, \gamma^{(n-1)}$
are the axial directions. 
As the inequality in \eqref{add4} becomes equality, we see that
$$  \lim_{\theta\to 0, s\to 0^+}  D^2_{\theta_i\theta_j} \zeta (\theta, s)=0 \ \ \text{for all}\  i\ne j. $$
By \eqref{add2} and since the inequality \eqref{add5} becomes equality,
we infer that 
\beq\label{ts00}
   \lim_{\theta\to 0, s\to 0^+}  \zeta_{\gamma^{(i)}\gamma^{(i)}} = c_{\gamma^{(i)}\gamma^{(i)}}. 
  \eeq
This completes the proof.
\end{proof}

%%% \begin{remark}
%%% In \S3, we use the partial Legendre transform to obtain the $C^{2,\alpha}$ regularity for the solution $\psi$ to \eqref{blow1}.
%%% But for equation \eqref{po1}, we cannot use the partial Legendre transform
%%% as it is defined on the sphere and also it contains lower order terms in the matrix.
%%% Our strategy is to use a blow-up argument and  the Bernstein theorem \ref{thmbern}
%%% to obtain the $C^2$ regularity.
%%% \end{remark}

\begin{remark} \label{r4.4} %%%Remark 4.2
In \S 3 and \S4, we employ the blow-up argument to prove the continuity of the second derivatives. 
The blow-up technique has been frequently used in geometric and analysis problems.
It usually contains two steps, one is to classify the blow-up limits and the other one is to show that
the limit is independent of choice of the blow-up sequences.
In the first step, one proves that a blow-up sequence  at a fixed point
converges along a sub-sequence to a nice limit. 
In the second step, one needs to prove that all blow-up sequences around a given point
converges to the same limit.
The second step is usually rather difficult. See \cite{Ca77, Ca98} for the classical obstacle problem \eqref{Lapob}.
There are many examples, for which one can prove the first step but the second one becomes extremely complicated, 
such as singularity profiles for the Ricci flow and the mean curvature flow,
and the $C^1$ regularity of infinity harmonic functions.
For infinity harmonic functions, 
the blow-up at a fixed point is an affine function  \cite{ES}, but the $C^1$ regularity in high dimensions is still open.

%%%Note that \eqref{ss00} and \eqref{ts00} are stronger than the uniqueness of the blow-up limits.  

%%% Consider the Monge-Amp\`ere equation \eqref{DP}.
%%% If the RHS $h$ is positive and continuous, 
%%% then the blow-up of $u$ at an interior point is a quadratic polynomial.
%%% But $D^2 u$ is in general not continuous if $h$ is not Dini-continuous.

%%% Our proof of Lemmas \ref{lemconti1} - \ref{lemconti3} relies on the special structure of the equation \eqref{po2}.
%%% The singular term $\frac{\zeta}{s}$ is very useful for our construction of the auxiliary functions $\sigma_{k,m}^{\pm}$.

\end{remark}

\vskip10pt

\section{\bf  Higher regularity of free boundary}

In this section we first establish a weighted $W^{2,p}$ estimate
for a linear singular elliptic equation of Keldysh type.
Then we prove the higher regularity of the free boundary.

\subsection{\bf  Linear singular elliptic equations of Keldysh type}
%section4

First we consider the Neumann problem of the singular elliptic equation with constant coefficients:
\begin{equation}\label{as1}
\begin{split}
&\mathcal L_b(u)=-\Delta u-b\frac{u_{x_n}}{x_n}=f \quad \text{in} \quad \mathbb R^n_+,\\
& u_{x_n}(x',0)=0 .
\end{split}
\end{equation}
We assume that   $b>0$ is a constant,
$f\in L^\infty({\mathbb R^n_+})$.
%%% Here we denote by $C_c^\infty(\overline{\mathbb R^n_+})$ smooth functions in $\overline{\mathbb R^n_+}$ with compact support.

Horiuchi  \cite{H95, H96} 
introduced the Green function for \eqref{as1} and proved a weighted Schauder type estimate.
%%% See also \cite{Chs} for the fundamental solution to the Keldysh equation.
In \cite{H95}, he found the following representation formula for the solution to \eqref{as1},
\begin{equation}\label{as2}
u(x)=\int_{\mathbb R^n_+} K_b(x,y)y_n^b f(y)dy=:T_b(f)(x) , 
\end{equation}
where $K_b(x,y)$ is the  Green function, given by 
\begin{equation}\label{G1}
\begin{split}
K_b(x,y) & =D_b \int_0^1 (|x-y|^2(1-\tau)+|x-y^*|^2\tau)^{-\frac{n-2+b}{2}}[\tau(1-\tau)]^{\frac b 2-1}d\tau, \\
       D_b & =2^{b-2}\pi^{-\frac n2}\frac{\Gamma(\frac{n+b-2}{2})}{\Gamma(\frac b2)},\quad y^*=(y_1,\cdots,y_{n-1},-y_n).
\end{split}
\end{equation}
Moreover, for  $b>0$, 
the following estimates hold, 
\begin{equation}\label{G2}
|\p_x^{\gamma_1}\p_y^{\gamma_2} K_b(x,y)|\le
\begin{cases}
C\, |x-y^*|^{-b}\ln \left(2+\frac{|x-y^*|}{|x-y|}\right)\ \ \  \text{if}\ n=2\ \text{and}\ |\gamma_1|=|\gamma_2|=0,\\ 
C \, |x-y^*|^{-b}|x-y|^{2-n-|\gamma_1|-|\gamma_2|},\quad \text{otherwise}
\end{cases}
\end{equation}
for any indexes $\gamma_1,\gamma_2\in \mathbb N^n$, where $C$ depends on $n, b, \gamma_1,\gamma_2$ (see Lemma 5-3 in \cite{H95}).

Horiuchi  \cite{H95} proved that  the function $u=T_b(f)$, given in \eqref{as2}, is a solution to \eqref{as1}.
By the representation formula \eqref{as2},
he also proved the following $C^{2, \alpha}$ estimate by careful computation as in \cite {GT}.
 
\begin{theorem}\label{T5.1} [Theorem 2-1,\cite{H95}]
Assume that $f\in C^\alpha(\overline{\mathbb R^n_+})$, supp$\, f\subset B_1$, and $b>0$.
Then 
\begin{equation} \label{e5a}
\sum_{i,j=1}^n\|\p_{ij}T_b(f)(x)\|_{C^{\alpha}(\overline{B_1^+})}
   +\Big \|\frac{\p_n T_b(f)(x)}{x_n}\Big\|_{C^{\alpha}(\overline{B_1^+})}\le C\|f\|_{C^{\alpha}(\overline{B_1^+})} , 
\end{equation}
where  $\alpha\in (0,1)$, and $C$ depends only on $n, b, \alpha$.
\end{theorem}

Since $u=T_b(f)$ and $f$ has compact support, from \eqref{e5a}  we also have
\begin{equation} \label{e5d}
\|u\|_{C^{2,\alpha}(\overline{\mathbb R^n_+})}+\|\frac{u_n}{x_n}\|_{C^\alpha(\overline{\mathbb R^n_+})}
            \le C\|f\|_{C^{\alpha}(\overline{R^n_+})} .
\end{equation}

Denote
\begin{equation*}
\|f\|_{L^p_{\mu_b}(\mathbb R^n_+)}=\Big(\int_{\mathbb R^n_+}|f|^p(x) x_n^bdx\Big)^{\frac 1p}
\end{equation*}
\begin{equation*}
\|f\|_{W^{k,p}_{\mu_b}(\mathbb R^n_+)}=\sum_{|\alpha|\le k}\|D^\alpha f\|_{L^p_{\mu_b}(\mathbb R^n_+)}.
\end{equation*}
We have the following weighted $W^{2,p}$ estimate.

\begin{theorem}\label{T5.2}
Let $u=T_b(f)$ be given by \eqref{as2}.   Assume that $f$ has compact support. 
Then  for any $p>1$ and  $b>1$, we have the estimate
\begin{equation} \label{e5b}
 {\Small\text{$ \sum_{i, j=1}^n $}} \|u_{x_ix_j} \|_{L^p_{\mu_b}(\mathbb R^n_+)}
      +\|\frac{u_n}{x_n}\|_{L^p_{\mu_b}(\mathbb R^n_+)}\le C\|f\|_{L^p_{\mu_b}(\mathbb R^n_+)}
\end{equation}
for a constant $C>0$ depending only on $p,n,b$.  
\end{theorem}

Note that \eqref{e5b} is invariant under dilation of coordinates. Hence to prove Theorem \ref{T5.2},
we may assume that $\text{supp} f\subset B_1(0)$. 
By Theorem \ref{T5.2}, we have

\begin{corollary}\label{C5.1}
Let $u=T_b(f)$ be given by \eqref{as2}.   Assume that $\text{supp} f\subset B_1(0)$.
Then  for any $p>\frac{n+b}{n-2+b}$ and  $b>1$, we have the estimate
\begin{equation} \label{e5c}
\|u\|_{W^{2,p}_{\mu_b}(\mathbb R^n_+)}+\|\frac{u_n}{x_n}\|_{L^p_{\mu_b}(\mathbb R^n_+)}\le C\|f\|_{L^p_{\mu_b}(\mathbb R^n_+)},
\end{equation}
where the constant $C$ depends on $p, n, b$.
\end{corollary}
 
\begin{remark}
(i)  The condition $b>1$ is used only in \eqref{ras1}, otherwise it suffices to assume $b>0$.
The condition $p>\frac{n+b}{n-2+b}$ in Corollary \ref{C5.1} is for the  estimate $\|u\|_{L^p_{\mu_b}}$.
\\
(ii) A different version of the $C^{2, \alpha}$ estimate in Theorem \ref{T5.1} was also obtained in \cite{HH12}.
The paper \cite{HH12} has also established a weighted $W^{2,p}$ estimate 
by the method of Fourier transformation and  oscillatory integral method
\cite [Theorem 1.2]{HH12}.  
However, the coefficient $a>\frac 32$ is needed in \cite{HH12} which corresponds to $b>2$ in our present case.  
Hence we can't apply \cite{HH12} to the problem in this paper.
 \end{remark}

In the following we prove Theorem \ref{T5.2}.
Recall that the $W^{2,p}$ estimate for the Laplace equation can be obtained
by the Calderon-Zygmund decomposition and the Marcinkiewicz interpolation \cite{GT}.
Our proof will adopt a similar idea. 
But due to the singular term $\frac{u_{x_n}}{x_n}$, we need to introduce a new variable in the proof
(see function $\tilde u$ in \eqref{utilde}).
We include the details of the proof for convenience of the reader.

First,  we notice the following asymptotic estimate for $T_b(f)$ near $\infty$.

\begin{lemma}\label{lemasy}
Let $u=T_{b}f$ be given by \eqref{as2}. Assume that $f\in L^\infty({\mathbb R^n_+})$,
supp$f\subset B_1(0)$, and  $b>0$.
Then we have the following asymptotic estimate,
\begin{equation}
|D^\gamma_x u(x)|\le \frac{C \|f\|_{L^1_{\mu_b}(\mathbb R^n_+)}}{|x|^{n-2+b+|\gamma|}}\ \ \forall\  |x|>2, 
\end{equation}
where  the constant $C$ depends only on $n, \gamma, b$. 
\end{lemma}

\begin{proof}
From \eqref{G2}, one has
\begin{equation}
\begin{split}
|D^\gamma_x u(x)|=&\Big|\int_{\mathbb R^n_+}\p_x^\gamma K_b(x,y)y_n^b f(y)dy\Big|\\
\le & C(\gamma,b)\int_{\mathbb R^n_+} |x-y^*|^{-b}|x-y|^{2-n-|\gamma|} f(y) y_n^b dy\\
\le &\frac{C_{\gamma,b}\|f\|_{L^1_{\mu_b}(\mathbb R^n_+)}}{|x|^{n-2+b+|\gamma|}}
\end{split}
\end{equation}
provided $|x|>2$.
\end{proof}
 
\begin{lemma}\label{leml2}
Assume that  $f\in L^\infty(\mathbb R^n_+)$,
supp\,$f\subset B_1$, and  $b>0$.
Then $u=T_{b}f$ satisfies the weighted $W^{2,2}$ estimate
\begin{equation}\label{int5}
\|D^2 u\|^2_{L^2_{\mu_b}(\mathbb R^n_+)}+b\Big \|\frac{u_n}{x_n}\Big\|^2_{L^2_{\mu_b}(\mathbb R^n_+)}= \|f\|^2_{L^2_{\mu_b}(\mathbb R^n_+)}.
\end{equation}
\end{lemma}

\begin{proof}
Since $u$ satisfies equation \eqref{as1}, we have 
\begin{equation}\label{int5a}
\int_{\mathbb R^n_+}\Big(\Delta u+b\frac{u_n}{x_n}\Big)^2 x_n^b dx=\int_{\mathbb R^n_+}f^2 x_n^b dx.
\end{equation}
By Lemma \ref{lemasy}, we can carry out integration by parts and obtain
$${\begin{split}
 \int_{\mathbb R^n_+}(\Delta_{x'} u)^2 x_n^b dx
      & = {\Small\text{$ \sum_{i,j=1}^{n-1} $}} \int_{\mathbb R^n_+}u^2_{x_ix_j} x_n^b dx, \\
 2\int_{\mathbb R^n_+} u_{ii} u_{nn} x_n^b dx
      & = -2\int_{\mathbb R^n_+} u_i u_{nni}x_n^b dx\\
      & = 2\int_{\mathbb R^n_+}u_{ni}^2 x_n^b dx+2b\int_{\mathbb R^n_+}u_iu_{ni} x_n^{b-1}dx,\ \  i=1,\cdots,n-1,
      \end{split}}
$$
and
 $${\begin{split}
 2b\int_{\mathbb R^n_+} u_{ii} u_{n} x_n^{b-1} dx
     & =-2b\int_{\mathbb R^n_+} u_i u_{ni}x_n^{b-1} dx,\ \ i=1,\cdots,n-1,\\ 
2b\int_{\mathbb R^n_+} u_{nn} u_{n} x_n^{b-1} dx
   & =-b(b-1)\int_{\mathbb R^n_+} u_n^2 x_n^{b-2} dx .
\end{split}}
$$
Summing up, and noticing that the two integrals $2b\int_{\mathbb R^n_+} u_i u_{ni}x_n^{b-1} dx$ are cancelled each other, 
and the last integral $ -b(b-1)\int_{\mathbb R^n_+} u_n^2 x_n^{b-2} dx$ is partly cancelled by
the left hand side of \eqref{int5a}, we obtain \eqref{int5}.
\end{proof}

By Lemma \ref{leml2}, we see that $\p_{ij} T_b$, $\frac{\p_n T_b}{x_n}$ are bounded linear operators 
from $L^2_{\mu_b}(\mathbb R^n_+)$ to  $L^2_{\mu_b}(\mathbb R^n_+)$. 
Next we show that 
they are bounded linear operators
from $L^1_{\mu_b}(\mathbb R^n_+)$ to $L^1_{\mu_b,\omega}(\mathbb R^n_+)$,
where 
\begin{equation}
L^1_{\mu_b,\omega}(\mathbb R^n_+) =\big \{v \text{ is measurable}:\ 
{\mu_b}\{x\in \mathbb R^n_+:\ |v(x)|>\alpha\}\le C\alpha^{-1}\ \forall \ \alpha>0 \big\} ,
\end{equation}
and the measure $\mu_b$ is defined as $\mu_b(E)=\int_{E}x_n^b dx$ for any measurable set $E$.

As a first step, we give the following weighted Calderon-Zygmund decomposition in $\mathbb R^n_+$.

\begin{lemma}\label{lemcz}
Suppose $f\in L^1_{{\mu_b}}(\mathbb R^n_+)$ and $b>0$. 
Then for any given constant $\alpha>0$, 
there is a sequence of disjoint   cubes $\{Q_k\}_{k=1}^\infty$ and a decomposition 
$$f=g+h=g+{\sum}_{k\ge 1}h_k$$ 
\vskip-15pt
\noindent
such that 
\begin{itemize}
\item[(i)]  $|g(x)|\le c\alpha$ for $\mu_b$-a.e. $x\in \mathbb R^n_+$, and 
\beq\label{p-i}
\int_{\mathbb R^n_+}|g(x)|d\mu_b\le c\int_{\mathbb R^n_+}|f(x)|d\mu_b.
\eeq
\item [(ii)] For each $k\ge 1$, $h_k$ is supported in $Q_k$ and satisfies
\beq\label{p-ii}
	\int_{\mathbb R^n_+} |h_k|d{\mu_b}\le c\alpha{\mu_b}(Q_k)\ \ \text{and}\ \  \int_{\mathbb R^n_+} h_k d{\mu_b}=0 .
\eeq
\item [(iii)] 
\beq\label{p-iii}
 {\sum}_k\,  {\mu_b}(Q_k)\le \frac{c}{\alpha}\int_{\mathbb R^n_+} |f|d{\mu_b} .
\eeq
Here the constant $c$ depends only on $n$ and $b$.
\end{itemize}
\end{lemma}

\begin{proof} 
The Calderon-Zygmund decomposition for the Lebesgue measure is well-known \cite{Ca89, GT}.
Here we replace the Lebesgue measure by the measure $\mu_b$ but the idea is the same.
Consider a  partition of $\mathbb R^n_+$ by cubes $\mathcal K_0=:\{Q\}$ whose side length $d$ is chosen such that
\begin{equation*}
\alpha>\frac 1{{\mu_b}([0,d]^n)}\int_{\mathbb R^n_+} |f|d{\mu_b}.
\end{equation*}
Then for any cube $Q\in\mathcal K_0$, we have $|f|_Q\le \alpha$, 
where $|f|_Q=\frac{1}{\mu_b(Q)}\int_Q |f| d\mu_b$.

To obtain the sequence of disjoint   cubes $\{Q_k\}_{k=1}^\infty$ stated in the lemma,
we first consider an arbitrary cube $Q\in\mathcal K_0$, 
by bisecting the edges of $Q$, we subdivide $Q$ into 
$2^n$ congruent sub-cubes $\{Q'_i\}_{i=1}^{2^n}$ with disjoint interiors.
For any sub-cube $Q'_i$, we have either $|f|_{Q'_i}\ge \alpha$ or $|f|_{Q'_i} < \alpha$.
Denote by $\mathcal K_1$ the set of all sub-cubes of side-length $d/2$.

\begin{itemize}
\item [(a)] If $|f|_{Q'_i}\ge \alpha$, set
$$h_{Q'_i}  =\chi_{{Q'_i}} (f-f_{Q'_i})\ \ \text{and}\ \ g_{Q'_i}=\chi_{{Q'_i}} f_{Q'_i}, $$
where $\chi$ is the characteristic function,
$f_{Q'_i}=\frac{1}{\mu_b(Q'_i)}\int_{Q'_i} f d\mu_b$ is the $\mu_b$-average of $f$ in $Q'_i$.
In this case we will not divide $Q'_i$ any more, 
and this $Q'_i$ will be counted as an element in the family $\{Q_k\}$.

\item [(b)]  If $|f|_{Q'_i} < \alpha$, we divide $Q'_i$ into $2^n$ equal sub-cubes $\{Q''_i\}$ as above,
and denote by $\mathcal K_2$ the set of all the sub-cubes of side-length $d/4$.
For any sub-cube $Q''_i\in \mathcal K_1$, 
if $|f|_{Q''_i}\ge \alpha$, we define the functions $h_{Q''_i}$ and $g_{Q''_i}$ as in (a) above,
and count $Q''_i$ as an element in the family  $\{Q_k\}$.
If $|f|_{Q''_i} < \alpha$, we divide $Q''_i$ into $2^n$ equal sub-cubes again. 
\end{itemize}

Repeating the above procedure indefinitely,  
we obtain a sequence of  disjoint  cubes $\{Q_k\}$ such that
$|f|_{Q_k}\ge \alpha$ for all $k\ge 1$. 
Now we define
	\begin{equation}\label{hig}
	{\begin{split}
	h_i & = h_{Q_i}\ \ i=1,2 \cdots, \\[4pt]
	g(x) &=\bigg \{\begin{split}
	& f_{Q_k}\ \ \ \text{if}\  x\in Q_k,\quad k=1,2,\cdots,\\
        & f(x)\ \ \text{if}\  x\in \mathbb R^n_+\backslash (\cup_{k=1}^{\infty} Q_k).
	\end{split}
	\end{split}}
	\end{equation}

For any cube $Q_k$ in the sequence, let $\tilde Q_k$ be its predecessor,
namely  $Q_k$ is one of the $2^n$ sub-cubes obtained from $\tilde Q_k$.
By the above decomposition, it is easy to see 
\beq\label{fq}
|f|_{Q_k}\le \frac{\mu_b(\tilde Q_k)}{\mu_b(Q_k)} |f|_{\tilde Q_k} \le 2^{n+2b+1} |f|_{\tilde Q_k} \le 2^{n+2b+1} \alpha. 
\eeq
One can verify that \eqref{p-i} - \eqref{p-iii} hold with $c=2^{n+2b+1}$. 
To finish the proof, it remains to show that $|g(x)|\le C\alpha \ \forall\ x\in \R^n_+$. 
Since the set $\cup_k \p Q_k$ has measure zero,
there is no loss of generality in assuming that $x$ is a Lebesgue point and 
$x\not\in \cup_k \p Q_k$.
Then in the procedure described in (a) and (b) above, 
eventually there will be a cube $Q_k$ such that  $x\in Q_k$ 
and $|f|_{Q_k}\ge \alpha$ if $g(x)>c\alpha$.
Hence by  our definition of $g$ in \eqref{hig}, we have 
$|g(x)| = |f_{Q_k}| \le C\alpha$ for $x\in Q_k$ by \eqref{fq}.
This completes the proof.
 \end{proof}

Next we show that 
$\p_{x_ix_j} T_b$ and $\frac{\p_n T_b}{x_n}$ are  bounded linear operators from $L^1_{\mu_b}(\mathbb R^n_+)$ 
to $L^1_{\mu_b,\omega}(\mathbb R^n_+)$, for any $b>0$.

\begin{lemma}\label{leml1}
Assume $f\in L^1_{\mu_b} (\mathbb R^n_+)$ and supp$\,f\subset B_1(0)$.
Then for any constant $\alpha>0$, $u=T_{b}(f) $ satisfies
\begin{equation} \label{mub}
{\begin{split}
{\mu_b} \{x\in \mathbb R^n_+:\ |u_{ij}(x)|>\alpha\}
  & \le C\alpha^{-1}\|f\|_{L^1_{\mu_b}(\mathbb R^n_+)} ,\\
 {\mu_b} \big\{x\in \mathbb R^n_+:\ \big|\frac{u_n}{x_n}(x)\big|>\alpha\big\}
  & \le C\alpha^{-1}\|f\|_{L^1_{\mu_b}(\mathbb R^n_+)} ,
   \end{split}}
\end{equation}
for a positive constant $C$ depending only on $n, b$.
\end{lemma}

\begin{proof}
Let $ f=g+h=g+\sum_{k}h_k$ be the decomposition given in Lemma \ref{lemcz}. 
Since $T_b$ is a linear operator, we have
\begin{equation}
\begin{split}
&{\mu_b}(\{x\in \mathbb R^n_+:\ |\p_{x_ix_j}T_b f(x)|>\alpha\})\\
    \le & {\mu_b}(\{x\in \mathbb R^n_+:\ |\p_{x_ix_j}T_b g(x)|>\frac{\alpha}{2}\})
       +  {\mu_b}(\{x\in \mathbb R^n_+:\|\p_{x_ix_j}T_b h(x)|>\frac{\alpha}{2}\}).
\end{split}
\end{equation}
By Lemma \ref{lemcz}, we have $|g|\le C\alpha$,
and  $g$ is compactly supported since $f$ is. 
By Lemma \ref{leml2}, one has 
\begin{equation}
\|\p_{x_ix_j}T_b g(x)\|_{L_{\mu_b}^2(\mathbb R^n_+)}^2\le  \|g\|_{L_{\mu_b}^2(\mathbb R^n_+)}^2\le C\alpha\|g\|_{L_{\mu_b}^1(\mathbb R^n_+)}\le C\alpha\|f\|_{L_{\mu_b}^1(\mathbb R^n_+)}.
\end{equation}
This implies
\begin{equation}
 {\mu_b} \big(\{x\in \mathbb R^n_+:\ |\p_{x_ix_j}Tg(x)|>\frac{\alpha}{2}\} \big)\le C\alpha^{-1}\|f\|_{L_{\mu_b}^1(\mathbb R^n_+)}.
\end{equation}

Now for any cube $Q_k$ in Lemma \ref{lemcz}, let $Q_k^*=(q+2(Q_k-q))\cap \mathbb R^n_+$,
where $q$ is the center of $Q_k$.
Then for any $x\in (Q_k^*)^c$,  there holds
\begin{equation}\label{Thk}
\begin{split}
|(\p_{x_ix_j}Th_k)(x)|
              &=  \Big|\int_{\mathbb R^n_+}\p_{x_ix_j}K_b(x,y) h_k(y)d{\mu_b}(y)\Big|\\
             &=  \Big|\int_{Q_k}\left[\p_{x_ix_j}K_b(x,y)-\p_{x_ix_j}K_b(x,q)\right] h_k(y)d{\mu_b}(y) \Big| \\ 
             &\le  \int_{Q_k}\left|\p_{x_ix_j}K_b(x,y)-\p_{x_ix_j}K_b(x,q)\right| | h_k(y)|d{\mu_b}(y) \\ 
             &\le C \int_{Q_k}\frac{|y_n-q_n||h_k(y)|}{|x-y^*|^{b}|x-y|^{n+1}}d{\mu_b}(y) ,
\end{split}
\end{equation}
where in the second equality, we used (ii) in Lemma \ref{lemcz}. 
By \eqref{Thk} we then have  
\begin{equation}
\begin{split}
&\int_{(Q_k^*)^c}|(\p_{x_ix_j}Th_k)(x)|d{\mu_b}(x)\\
 &\le  C \int_{Q_k}|h_k(y)|d{\mu_b}(y) \int_{(Q_k^*)^c}\frac{|y_n-q_n|\, d{\mu_b}(x)}{|x-y^*|^b|x-y|^{n+1}}\\
 &\le  C\int_{Q_k}|h_k(y)|d{\mu_b}(y)\le C\alpha {\mu_b}(Q_k) ,
\end{split}
\end{equation}
where $y^*$ is reflection of $y$ in $\{x_n=0\}$.
Then
\begin{equation*}
\begin{split}
&{\mu_b}(\{x\in \mathbb R^n_+:\ |\p_{x_ix_j}Th(x)|>\frac{\alpha}{2}\})\\
\le & {\sum}_k\, {\mu_b}(Q_k^*)+\frac{2}{\alpha}{\sum}_{k}\int_{(Q_k^*)^c}|(\p_{x_ix_j}Th_k)(x)|d{\mu_b}(x)\le \frac{C\|f\|_{L_{\mu_b}^1(\mathbb R^n_+)}}{\alpha}.
\end{split}
\end{equation*}
We obtain the first inequality in \eqref{mub}.

For the second inequality in \eqref{mub}, 
recall that $T_bf$ satisfies equation \eqref{as1}. 
Hence by \eqref{as1}, we infer that $\frac{\p_n T_b}{x_n}$ 
is also a bounded linear operator from $L^1_{\mu_b}(\mathbb R^n_+)$ 
to $L^1_{\mu_b,\omega}(\mathbb R^n_+)$.
\end{proof}

With the above preparation, we are ready to prove Theorem \ref{T5.2}

\vskip10pt

\noindent \textit{Proof of Theorem \ref{T5.2}.}
By Lemma \ref{leml2} and Lemma \ref{leml1}, the operator $\p_{x_ix_j}T_b$  satisfies
\begin{equation}
{\begin{split}
  & \|\p_{x_ix_j}T_b(f)\|_{L^2_{\mu_b}(\mathbb R^n_+)} \le C_{n,b}\|f\|_{L^2_{\mu_b}(\mathbb R^n_+)},\\
  &  \|\p_{x_ix_j}T_b(f)\|_{L^1_{{\mu_b},\omega}(\mathbb R^n_+)} \le C_{n,b}\|f\|_{L^1_{\mu_b}(\mathbb R^n_+)} 
\end{split}}
\end{equation}
for any $b>0$, where 
$$ \|v\|_{L^1_{{\mu_b},\omega}(\mathbb R^n_+)}
    =: \sup_{\alpha>0}\, \alpha\mu_b\{x\in \R^n_+:\ |v|>\alpha\}. $$
   Hence by the Marcinkiewicz Interpolation   [Theorem 1.3.1,\cite{BL12}], we infer that
\begin{equation}
\|\p_{x_ix_j}T_b(f)\|_{L^p_{\mu_b}(\mathbb R^n_+)}\le C_{n,b,p}\|f\|_{L^p_{\mu_b}(\mathbb R^n_+)}
\end{equation}
for any $p\in (1,  2]$, and for all $1\le i, j\le n$.
In applying the Marcinkiewicz Interpolation,
we regard $\p_{x_ix_j}T_b$ as the mapping in \cite{BL12}.

When $p>2$, let $p'=\frac{p}{p-1}<2$.   
For any $f, g\in C_c^\infty(\overline{\mathbb R^n_+})$, 
noting that the kernel $K_b(x, y)$ in \eqref{G1} is symmetric in $x'$ and $y'$, 
we have
\begin{equation}\label{301}
\int_{\mathbb R^n_+}  g \p_{ij}T_b (f)d{\mu_b}(x)=\int_{\mathbb R^n_+} f\p_{ij}T_b (g)d{\mu_b}(x) \ \ 
\forall\ 1\le i,j\le n-1 .
\end{equation}
Here we denote by $C_c^\infty(\overline{\mathbb R^n_+})$ smooth functions in $\overline{\mathbb R^n_+}$ with compact support. %%% The identity \eqref{301} comes from that $K_b(x,y)$ depends only on $|x'-y'|,x_n,y_n$. 
Hence we have
\begin{equation}\label{estlp1}
\begin{split}
\|\p_{ij}T_b (g)\|_{L^p_{\mu_b}(\mathbb R^n_+)}
  &=\sup_{\|f\|_{L^{p'}_{\mu_b}(\mathbb R^n_+)}=1}\int_{\mathbb R^n_+}f\p_{ij}T_b (g)d{\mu_b}(x)\\
  &\le \sup_{\|f\|_{L^{p'}_{\mu_b}(\mathbb R^n_+)}=1}
             \|g\|_{L^p_{\mu_b}(\mathbb R^n_+)}\|\p_{ij}T_b (f)\|_{L^{p'}_{\mu_b}(\mathbb R^n_+)}\\
  & \le C_{p'}\|g\|_{L^p_{\mu_b}(\mathbb R^n_+)} \ \ \forall\ 1\le i, j\le n-1.
\end{split}
\end{equation}
By approximation, we see that \eqref{estlp1} holds for non-smooth function $g$.

To prove the estimate \eqref{estlp1} for $\p_{in} T_b(g)$, $i=1,\cdots,n$, 
we need to introduce a new function $\tilde u$ defined in $\R^{n+1}$ (raising the dimension to $n+1$).   
Denote $u(x)=T_b(g)$ and  
\beq\label{utilde}
\tilde u(x,x_{n+1})=u\Big(x',\sqrt{x_n^2+x_{n+1}^2}\, \Big) = u(x', r),
\eeq
where $(x,x_{n+1})\in \mathbb R^{n+1}_+$, $r=\sqrt{x_n^2+x_{n+1}^2}$.
By direct computation, we have
\begin{equation}
{\begin{split}
  & \tilde u_{x_i}(x', x_n, x_{n+1})=u_{x_i}(x', r),\quad i=1,\cdots,n-1, \\
  & \frac{ \tilde u_{x_n}(x', x_n, x_{n+1})}{x_n} = \frac{u_r(x', r)}{r},\\
  & \frac{\tilde u_{x_{n+1}}(x', x_n, x_{n+1})}{x_{n+1}}=\frac{u_r(x', r)}{r},
  \end{split}}
\end{equation}
Hence $\tilde u$ satisfies
\begin{equation}\label{ras1}
\begin{cases}
\Delta_{x,x_{n+1}}  \tilde u+  \frac{b-1}{x_{n+1}}\tilde u_{n+1} =\tilde g\ \ \text{in}\  \mathbb R^{n+1}_+,\\
\tilde u_{n+1}(x,0)=0 \ \ \forall\  x\in \mathbb R^n  , 
\end{cases}
\end{equation}    
where $\tilde g(x,x_{n+1})=g(x', r)$, $\tilde u_{n+1} :=\tilde u_{x_{n+1}}$.
The boundary condition $\tilde u_{n+1}(x,0)=0 $ is due to that $u$ is even in $x_{n+1}$.
For equation \eqref{ras1}, we need to assume that $b>1$ for the $W^{2,p}$ estimate.

As $\tilde u$ and $T_{b-1}(\tilde g)$ satisfy the same equation \eqref{ras1}, 
by the Neumann boundary condition and the decay at $\infty$, we infer that
$\tilde u=T_{b-1}(\tilde g)$.
Hence for $b>1$ and $i=1,\cdots,n$, similarly to estimate \eqref{estlp1}, we have
\begin{equation}
{\begin{split} 
\int_{\mathbb R^{n+1}_+}|\tilde u_{ni}|^p x_{n+1}^{b-1} dxdx_{n+1}
  & \le C_{p,b} \int_{\mathbb R^{n+1}_+}|\tilde g|^p x_{n+1}^{b-1} dxdx_{n+1} \\
  & =\tilde C_{p,b} \int_{\mathbb R^{n}_+}|g|^p x_{n}^{b} dx.
  \end{split}}
\end{equation}
For the equality above, we use the polar coordinates for the variables $x_n, x_{n+1}$, i.e. 
$x_n=r\cos\theta$, $x_{n+1}=r\sin\theta$, $\theta\in [0,\pi]$. 
Notice that
\begin{equation*}
\begin{split}
&\int_{\mathbb R^{n+1}_+}|\tilde u_{ni}|^p x_{n+1}^{b-1} dxdx_{n+1}\\
=&\int_{\mathbb R^{n+1}_+}\left|u_{ri}\cdot \frac{x_n}{r}\right|^p x_{n+1}^{b-1}dxdx_{n+1}\\
=& \int_{\mathbb R^n_+}|u_{ri}|^p r^bdx'dr\cdot\int_0^\pi |\cos\theta|^p(\sin\theta)^{b-1} d\theta\\
=&\overline{C}_{p,b}\int_{\mathbb R^n_+} |u_{ni}|^p x_{n}^{b} dx.
\end{split}
\end{equation*}
We obtain
\begin{equation}\label{Tbij}
\|\p_{ij}T_{b}g\|_{L^p_{\mu_b}(\mathbb R^n_+)} \le C_{b,n,p}\|g\|_{L^p_{\mu_b}(\mathbb R^n_+)},\quad p\in (1,+\infty)
\end{equation}
for all $1\le i,j\le n$.  

\vskip10pt

By equation \eqref{as1} and  estimate \eqref{Tbij}, we also have
\begin{equation} 
\|\frac{u_n}{x_n}\|_{L^p_{\mu_b}(\mathbb R^n_+)}\le C\|f\|_{L^p_{\mu_b}(\mathbb R^n_+)} .
\end{equation}
This completes the proof of Theorem \ref{T5.2}. \hfill$\square$

 \vskip10pt

\noindent \textit{Proof of Corollary \ref{C5.1}.}
By Theorem \ref{T5.2} and the interpolation inequality, it suffices to prove
\begin{equation}\label{uLp}
\|u\|_{L^p_{\mu_b}(\mathbb R^n_+)}\le C\|f\|_{L^p_{\mu_b}(\mathbb R^n_+)}.
\end{equation}
By a dilation of the coordinates,  we assume that $\text{supp}\, f\subset B_1(0)$.
By Lemma \ref{lemasy}, we have
\begin{equation}
|u(x)|\le \frac{C\|f\|_{L^1_{\mu_b}(\mathbb R^n_+)}}{|x|^{n-2+b}}
        \le \frac{C\|f\|_{L^p_{\mu_b}(\mathbb R^n_+)}}{|x|^{n-2+b}},\quad |x|\ge 2.
\end{equation}
It follows that
\begin{equation}\label{uLp1}
{\begin{split}
\int_{\mathbb R^n_+\backslash B_2}|u|^p d\mu_b
        & \le C^p\|f\|^p_{L^p_{\mu_b}(\mathbb R^n_+)}\int_{\mathbb R^n_+\backslash B_2}\frac{x_n^b}{|x|^{p(n-2+b)}}dx \\
        & \le C_1C^p\|f\|^p_{L^p_{\mu_b}(\mathbb R^n_+)}
        \end{split}}
\end{equation}
provided $p(n-2+b)>n+b$.  

Let $\xi \in C^\infty(\mathbb R^n)$ be a  cut-off function  satisfying
$0\le \xi\le 1$ in $\R^n$, $\xi = 1$  in $B_2$, and $\xi=0$ outside $B_4$. 
Then by Poincare's inequality, we have
\begin{equation}
\begin{split}
\int_{\mathbb R^n_+}|\xi u|^pd\mu_b
   & \le  C\int_{\mathbb R^n_+}|D^2(\xi u)|^pd\mu_b\\ 
   & \le C\int_{B_4\cap \mathbb R^n_+}\left(|D^2 u|^p+|D^2\xi|^p |u|^p \right)d\mu_b\\
   & \le C\|f\|^p_{L^p_{\mu_b}(\mathbb R^n_+)}.
\end{split}
\end{equation}
Combining the above two estimates yields \eqref{uLp}.
\hfill$\square$

We have obtained the $C^{2,\alpha}$ and $W^{2,p}$ estimates for the special solution $u=T_b(f)$.
Next we prove that these two a priori estimates hold for any other solutions to \eqref{as1}.
First we prove

\begin{lemma}\label{L5.5}
Let $u\in W^{2,p}_{loc}(\R^n_+)\cap  C^1 (\overline{\R^n_+})$ be a solution to
\beq\label{as1a}
-\Delta u-b\frac{u_{x_n}}{x_n} = f \ \ \ \text{in}\ \ \R^n_+ .
\eeq
Assume  $b>1$, $p>n+b$,  and  $f\in L^p_{\mu_b} (B_1^+)$. 
 Then $u$ satisfies 
 \beq\label{as1b}
 u_{x_n} =0\ \ \ \text{on}\ \ x_n=0 .  
\eeq
\end{lemma}

\begin{proof}
 Let $\eta(\tau)\in C_c^\infty(\mathbb R)$ be a cut-off function which satisfies
 $0\le \eta\le 1$, $\eta\equiv 1$ in  $B_1(0)$, and  supp\,$\eta\subset B_2(0)$.
Let 
$
\eta_{\ve}(x)=\eta\left(\frac{|x|}{\ve}\right)
$
for a small constant $\ve>0$, 
and 
let $v=\eta_\ve u$. Then $v$ solves the following equation
\begin{equation*}
\mathcal R(v)=:\ \Delta v+b\frac{v_n}{x_n} =\eta_\varepsilon f + 2\partial_i u\partial_i \eta_\varepsilon+u\Delta \eta_\varepsilon
                        + bu \frac{\partial_n\eta_\varepsilon}{x_n}  =:\hat f . 
\end{equation*}
Denote $W= T_b(\hat f) $,
where $T_b(\cdot)$ is the integral operator given by \eqref{as2}. 
Denote $H(x_n)=x_n^{1-b}$. 
Then $\mathcal R(H)= 0$ and $H\to \infty$ as $x_n\to 0^+$.  Hence
\begin{equation*}
\mathcal R(v-W+\epsilon H)=0\quad \text{in}\quad \mathbb R^n_+.
\end{equation*}
By Lemma \ref{lemasy}, $|W(x)|\le C/|x|^{n-2+b} = o(H)$ as $|x|\to\infty$. Hence
\begin{equation*}
\lim_{x\rightarrow \infty}(v-W+\epsilon H)\ge 0
\end{equation*}
for any given small $\eps>0$.
By Theorem \ref{T5.2}, $W\in W^{2,p}_{\mu_b}(\R^n_+)$. 
Hence $W\in L_{loc}^\infty(\R^n_+)$ by the Sobolev embedding  \cite[Lemma B.3]{HH12}.
We obtain
\begin{equation*}
 \lim_{x_n\rightarrow 0}(v-W+\epsilon H)\ge 0.
\end{equation*}
Moreover, since $W\in W^{2,p}_{\mu_b}(\R^n_+)$ for $p>n$ and $v\in C^2(B_1^+)$, 
we can apply Aleksandrov's maximum principle for strong solution \cite{GT} and obtain
$
v-W\ge -\epsilon H.
$
Letting $\epsilon\rightarrow 0$, we obtain $v\ge W$. 
Similarly, we have $W\ge v$. 
This implies 
\begin{equation} \label{uvw}
\eta_\ve u=v=W=T_b(\hat f)
\end{equation}
and so \eqref{as1b} is proved.
\end{proof}
 
The condition $u\in  C^1 (\overline{\R^n_+})$ in Lemma \ref{L5.5} is such that $u_{x_n}$ exists on $x_n=0$.
If $p>n+b$ and $u\in W^{2,p}_{loc}(\overline{\R^n_+})$, 
by the Sobolev embedding we have $u\in  C^1 (\overline{\R^n_+})$.
For the proof of \eqref{uvw}, it suffices to assume that $u\in W^{2,p}_{loc}(\R^n_+)\cap  L^\infty (\R^n_+)$.
When $b\ge 1$ is an integer, $u_{x_nx_n}+\frac{b}{x_n} u_{x_n}$ is actually the Laplacian operator 
for rotationally symmetric functions in $\R^{1+b}$. 
Lemma \ref{L5.5} means for bounded solutions, the singularity at $x_n=0$ is removable.

By \eqref{uvw}, we see that the $C^{2,\alpha}$ and $W^{2,p}$ estimates in Theorems \ref{T5.1} and \ref{T5.2} 
hold for any solutions  $u\in W^{2,p}_{loc}(\R^n_+)\cap  L^\infty (\R^n_+)$.
Moreover, the estimates can be extended easily to linear singular elliptic equations of Keldysh type with variable coefficients,
\begin{equation}\label{new1}
 {\Small\text{$\sum$}}_{i,j=1}^{n}a^{ij}\partial_{ij} u
    + {\Small\text{$\sum$}}_{i=1}^{n-1}b^i\partial_i u+\frac{b^n}{x_n}\partial_n u+cu=f\ \ \ \text{in}\ B_1^+ .
\end{equation}
Assume that
\begin{equation}\label{new2}
\begin{split}
 & 0<\lambda I \le (a^{ij})_{i,j=1}^n\le\Lambda I<+\infty \ \ \  \text{in}\ B_1^+ ,\\
 & \frac{b^n}{a^{nn}}=b>1\ \ \text{is a constant}, \\
 & |c|+{\Small\text{$\sum$}}_{i=1}^n |b^i|\le \Lambda\ \ \ \text{in}\ B_1^+,
\end{split}
\end{equation}
for two positive constants $\lambda$, $\Lambda$. 
Then by the freezing coefficient method, we have the following a priori estimates.

\begin{theorem}\label{T5.3}
Let $u\in W^{2,p}_{loc}(\R^n_+)\cap  L^\infty (\R^n_+)$ be a solution to  \eqref{new1}.
Assume that $a^{ij}\in C(\overline{B_1^+})$ satisfy conditions \eqref{new2},
and $f\in L^p_{\mu_b} (B_1^+)$ for some $p>n+b$.
Then $u\in W^{2,p}_{\mu_b}(B_{1/2}^+)$ and satisfies the estimate
\begin{equation}\label{est-fullp}
\|u\|_{W^{2,p}_{\mu_b}(B_{1/2}^+)}+\Big\|\frac{u_n}{x_n}\Big\|_{L^p_{\mu_b}(B_{1/2}^+)}
      \le C\left(\|f\|_{L^p_{\mu_b}(B_1^+)}+\|u\|_{L^p_{\mu_b}(B_1^+)}\right) ,
\end{equation}
where  $C>0$ depends only on $p, n, \lambda, \Lambda$ and the modulus of continuity of $a^{ij}$.
\end{theorem}

\begin{theorem}\label{T5.4}
Let $u\in C^2(B_1^+)\cap C^{1,\alpha}(\overline{B_1^+})$ be a solution to  \eqref{new1}.
Assume that $a^{ij},b^n,c,f\in C^\alpha(\overline{B_1^+})$ and condition \eqref{new2} holds.
 Then $u\in C^{2,\alpha}(\overline{B_{1/2}^+})$ and we have the estimate 
 \begin{equation}\label{544}
\|u\|_{C^{2,\alpha}(\overline{B_{1/2}^+})}+\Big\|\frac{u_n}{x_n}\Big\|_{C^\alpha(\overline{B_{1/2}^+})}
             \le C\left(\|f\|_{C^{\alpha}(\overline{B_1^+})}+\|u\|_{C^0(\overline{B_1^+})}\right)
\end{equation}
for a constant $C>0$ depending only on $n,\alpha,\lambda,\Lambda$ and $\|a^{ij}\|_{C^{\alpha}(\overline{B_1^+})},\|b^i\|_{C^{\alpha}(\overline{B_1^+})}$, $\|c\|_{C^{\alpha}(\overline{B_1^+})}$.
\end{theorem}

\vskip10pt

\subsection{\bf Smoothness of  free boundary}

By the $C^{2,\alpha}$ and $W^{2,p}$ estimates in Section 5.1,  
we can prove  the higher regularity of the tangent cone of the solution $u$ to \eqref{MA-1}.
As before we assume that $u(0)=0$ and $u(x)>0$ $\forall\ x\ne 0$.

\begin{theorem}\label{thm3}
Let $\phi$ be the tangential cone of $u$ at $0$.
Then  the section $S_{1,\phi}=\{x\in \R^n:\ \phi(x)<1\}$ is uniformly convex and $C^{\infty}$ smooth 
if $g$ is positive and $C^\infty$ smooth.
\end{theorem}

To prove Theorem \ref{thm3},  by the definition of $\zeta$ in \eqref{zeta0}, 
it suffices to prove that  $\zeta(\theta,0)\in C^\infty(\mathbb S^{n-1})$.
We have the following stronger result, which implies Theorems \ref{thm3} and  \ref{thmB}.

\begin{theorem}\label{thm4} 
Let $\zeta(\theta,s)\in C^{2}(\mathbb S^{n-1}\times [0,1))$ be a solution to \eqref{po2}.
Assume that $\zeta_{s}(\theta,0)=0$
and  $\bar g$ is positive and smooth.
Then $\zeta(\theta,s)\in C^{\infty}(\mathbb S^{n-1}\times[0,1))$.
\end{theorem}

 \begin{proof} 
Differentiating \eqref{po2} with respect to $\theta_k$, one gets
\begin{equation}\label{high1}
\mathcal L(V)=V_{ss}+\frac{n+2}{n}\frac {V_s}{s}
  + {\Small\text{$ \sum_{i,j=1}^{n-1} $}} a^{ij} V_{\theta_i\theta_j} + {\Small\text{$ \sum_{i=1}^{n-1}  $}} a^{in} V_{\theta_is}=h
\end{equation}
where $V=\zeta_{\theta_k}$.
To apply the a priori estimate in \S5.1 to \eqref{high1}, we express equation \eqref{high1}
in a {\it local coordinates} on $\mathbb S^{n-1}$.
By Theorem \ref{thm2},  $a^{ij}$ and $h$ are continuous in $\theta,s$. 
By Lemma \ref{unif-ellip}, the operator $\mathcal L$ is uniformly elliptic.
Hence all the assumptions in Theorem \ref{T5.3} are fulfilled for $V$.  
Letting  $p>n+b$ and
by the Sobolev embedding, $W^{1,p}_{\mu_b}\to C^\alpha$ for some $\alpha>0$ (Lemma B.3, \cite{HH12}),
we have $V=\zeta_{\theta_k}\in C^{1,\alpha}(\mathbb S^{n-1}\times [0,1))$. 
 
Write equation \eqref{po2} in the form
\begin{equation}\label{ode1}
\zeta_{ss}+\frac{n+2}{n}\frac{\zeta_s}{s}=\tilde f\in C^\alpha(\mathbb S^{n-1}\times [0,1)) ,
\end{equation}
where $\tilde f$ is a smooth function of  $\theta, s, \zeta,  D \zeta,  D \zeta_\theta$.
Hence $\tilde f$ is H\"older continuous in $s, \theta$.
The solution to \eqref{ode1} is given by
\beq \label{ode2}
\zeta(\theta, s)=\zeta(\theta, 0)+\int_0^s  r^{-\frac{n+2}{n}}\int_0^r \lambda^{\frac{n+2}{n}} \tilde f(\theta,\lambda)d\lambda. 
\eeq
Hence we have  
\beq\label{ode3}
\begin{split}
 \zeta_s(\theta,s) 
    &=s^{-\frac{n+2}{n}}\int_0^s \lambda^{\frac{n+2}{n}} \tilde f(\theta,\lambda)d\lambda,\\
 \zeta_{ss}(\theta,s) 
    &=-\frac{n+2}{n}s^{-\frac{n+2}{n}-1}\int_0^s \lambda^{\frac{n+2}{n}} \tilde f(\theta,\lambda)d\lambda
                                 +\tilde f\in C^\alpha(\mathbb S^{n-1}\times[0,1)). 
\end{split}
\eeq
This implies $\zeta\in C^{2,\alpha}(\mathbb S^{n-1}\times [0,1))$. 
Hence  the coefficients $a^{ij}, h \in C^{\alpha}(\mathbb S^{n-1}\times [0,1))$.  

By the H\"older continuity of the coefficients,
we can obtain the $C^{2,\alpha}$ regularity of $ D_{\theta}^k \zeta$ for all $k\ge 1$.
Indeed, differentiating equation \eqref{high1} in $\theta$, and by Theorem \ref{T5.4},
we infer that $\zeta_{\theta_k}\in C^{2,\alpha}(\mathbb S^{n-1}\times [0,1))$. 
Differentiating equation \eqref{high1} in $\theta$ again, and also by Theorem \ref{T5.4},
we have $ D_{\theta}^2 \zeta\in C^{2,\alpha}(\mathbb S^{n-1}\times [0,1))$.
Repeating the argument we obtain
$ D_{\theta}^k \zeta\in C^{2,\alpha}(\mathbb S^{n-1}\times [0,1))$ for all integers $k\ge 0$.

%%%  \begin{remark}
%%%  (i) By the higher order regularity of $\zeta$ in $\theta$, we obtain Theorem \ref{thm3}. 
%%% Hence we obtain the higher order regularity of the free boundary.\\
%%% (ii) We can use \eqref{ode3} to obtain the higher regularity of $\zeta$ in $s$.
%%% Indeed, differentiating \eqref{ode3} in $\theta$, we obtain that 
%%% $ D_{\theta}^k \zeta_{ss} \in C^{\alpha}(\mathbb S^{n-1}\times [0,1))$ for all integers $k\ge 0$.
%%% Hence $ D^k_\theta\zeta_s\in  C^{1,\alpha}(\mathbb S^{n-1}\times [0,1))$.
%%% It implies that $ D^k_\theta\tilde f \in  C^{1,\alpha}(\mathbb S^{n-1}\times [0,1))$ for all $k\ge 0$.
%%% Regarding $\theta$ as parameters and using the Taylor expansion and \eqref{ode3},
%%% we can obtain higher regularity of $\zeta$ in $s$.
%%% But to obtain the expansion \eqref{tur}, we can use Theorem \eqref{T5.4} to prove the regularity of $u$ 
%%% in $s^2=r^n$ as follows (recall that $s=r^{n/2}$ and $r=|x|$).
%%% \end{remark}

To prove the higher order regularity of $\zeta$ in $s$  and the expansion \eqref{tur}, 
let  $t=\frac {s^2}4$. Then equation \eqref{po2} changes to equation  \eqref{max2}. 
Differentiating \eqref{max2} in $t$ and letting $V=\zeta_t$,  one gets
\begin{equation}\label{high-a3}
\begin{split}
\mathcal L(V)=tV_{tt}+\frac{2n+1}{n} V_{t}+ {\Small\text{$ \sum_{i,j=1}^{n-1} $}} a^{ij}V_{\theta_i\theta_j}
+t {\Small\text{$ \sum_{i=1}^{n-1} $}}  \tilde a^{ij} V_{t\theta_i}=h
\end{split}
\end{equation}
where $a^{ij}, \tilde a^{ij},h$ are smooth as functions of
$\theta,t,\zeta, D_{\theta,t} \zeta,  D^2_{\theta}\zeta, \zeta_{t\theta},t\zeta_{tt}.$
Note that the operator $\mathcal L$ in \eqref{high-a3} is different from that in \eqref{max3}.
We have moved some terms to the right hand side.

By the regularity $\zeta(\theta,s),\zeta_{\theta}(\theta,s)\in C^{2,\alpha}(\mathbb S^{n-1}\times [0,1))$ 
and the relation $t=\frac{s^2}{4}$, we have $ D^2_{\theta}\zeta, \zeta_{t\theta},t\zeta_{tt}\in C^\alpha(\mathbb S^{n-1}\times [0,1))$ in variables $\theta,t$. Hence, $a^{ij}, \tilde a^{ij},h\in {\color{red}C^{\alpha}}$ in variables $\theta,t$.

To apply  Theorem \ref{T5.4} to equation \eqref{high-a3},
we need to change the H\"older norm from the variable $s$ to $t=s^2/4$.  
Hence we denote
$${\begin{split}
 & \|f\|_{\widetilde C^{\alpha}(\overline{B^+_1})}
   = \|f\|_{L^\infty(B_1^+)} +\sup_{x\ne y\in B_1^+}\frac{|f(x)-f(y)|}{(|x'-y'|^2+(\sqrt{x_n}-\sqrt{y_n})^2)^{\frac \alpha2}} ,\\
 &   \|f\|_{\widetilde C^{k, \alpha}(\overline{B^+_1})}
   =   \|f\|_{C^k(\overline{B^+_1})} +\sup_{x\ne y\in B_1^+,|\beta|=k}\frac{|D^\beta f(x)-D^\beta f(y)|}{(|x'-y'|^2+(\sqrt{x_n}-\sqrt{y_n})^2)^{\frac \alpha2}} .
\end{split}} $$
Then estimate \eqref{544} can be reiterated as 
\begin{equation}\label{high4}
\begin{split}
 \|tu_{tt}\|_{\widetilde C^\alpha(\overline{B_{1/2}^+})}
 &+\|t^{\frac 12} D_\theta u_{t}\|_{\widetilde C^\alpha(\overline{B_{1/2}^+})}
  +\| D^2_\theta u\|_{\widetilde C^\alpha(\overline{B_{1/2}^+})}\\
 &+\|u\|_{\widetilde C^{1,\alpha}(\overline{B_{1/2}^+})}
    \le C\left(\|f\|_{\widetilde C^\alpha(\overline{B_1^+})}+\|u\|_{\widetilde C^\alpha(\overline{B_1^+})}\right).
\end{split}
\end{equation}
Applying \eqref{high4} to equation \eqref{high-a3}, we obtain
$t\partial_t^3\zeta,\ \partial_\theta^2\partial_t\zeta, \ t^{1/2}\partial_t^2\partial_\theta\zeta
\in \widetilde C^{\alpha}(\mathbb S^{n-1}\times [0,1))$, and
$\zeta_t\in \widetilde C^{1,\alpha}(\mathbb S^{n-1}\times [0,1))$.
Differentiating \eqref{high-a3} in $t$ repeatedly and using estimate \eqref{high4}, we obtain
$\zeta \in C^\infty(\mathbb S^{n-1}\times [0,1))$.

In the above we have shown that $\zeta$ is smooth in $t$.
Recall that $t=\frac{s^2}4= \frac{r^n}{4}$, and $\zeta=\frac ur$. 
Hence we obtain the Taylor expansion \eqref{tur}.
This completes the proof of Theorem \ref{thm4} and also that of Theorem \eqref{thmB}.
\end{proof}

\section{\bf Analyticity of the free boundary}
%section5

In this section, we prove the analyticity of the free  boundary $\Gamma$. 
Let $u$ be the solution to \eqref{MA-1}.
Let $\zeta=\frac ur$  and $s=r^{n/2}$ as in \S4, so that $\zeta$ satisfies equation \eqref{po2}.
Then the analyticity of the free boundary is equivalent to showing $\zeta(\theta,0)\in C^\omega(\mathbb S^{n-1})$.  
Here $C^\omega$ denotes the set of analytic functions. 
As before we assume that $u(0)=0$ and $u(x)>0$ $\forall\ x\ne 0$.
We have the following result.

\begin{theorem}\label{thm6} 
Let $\zeta(\theta,s)$ be a solution to \eqref{po2}.
Assume that $\bar g$ is positive and analytic. 
Then $\zeta(\cdot,s)\in C^{\omega}(\mathbb S^{n-1})$, for any $s\in [0,1)$.
\end{theorem}
 
To prove the analyticity of a function $u$, one needs to control the growth rate of its derivatives.
That is, for any multi-index $\alpha=(\alpha_1, \cdots, \alpha_n)$,  we need to prove
\beq \label{grow1}
 |\p^\alpha u|\le  C A^{|\alpha|} \alpha! 
 \eeq
for  sufficiently large constants $C, A$, independent of $\alpha$, 
where $|\alpha|=\alpha_1+\cdots +\alpha_n$, $\alpha!=\alpha_1!\cdots\alpha_n!$,
and $\partial^\alpha=\partial_{x_1}^{\alpha_1}\cdots\partial_{x_n}^{\alpha_n}$.

We will use the local coordinate system \eqref{sphcor1} in a neighbourhood of $\theta=0$.
For simplicity of notations,  
we use $x'=(x_1, \cdots, x_{n-1})$ to denote $\theta=(\theta_1, \cdots, \theta_{n-1})$, and
use $x_n$ to denote $s$. 
Then in the coordinates  \eqref{sphcor1},  equation \eqref{po2} can be written in the form 
\begin{equation}\label{eqana1}
\zeta_{nn} + b\frac{\zeta_n}{x_n} + 
F(x, \zeta,  \p_{x_i} \zeta,  \p^2_{x_ix_j} \zeta)=0, \quad (x',x_n)\in Q_{r_0} ,
\end{equation}
where $b=\frac{n+2}n$, $Q_{r_0}= B'_{r_0}(0)\times [0, 1]$,  and $B'_{r_0}(0)=\{|x'|<r_0\}$ is a ball in $\R^{n-1}$.

The function $F(x, z, p,  r)$ is defined for  
$x\in   Q_{r_0} $, $z\in\R^1$, $p\in\R^n$, and $r=(r_{ij})\in \mathbb S^{n\times n}$ but is independent of $r_{nn}$.
As a function, $F$  is analytic in its arguments.
By Lemma \ref{unif-ellip}, equation \eqref{eqana1} is uniformly elliptic.
  
\vskip5pt

The analyticity of solutions to uniformly elliptic equations has been studied by many people \cite{F58, Mo58}.
A simple proof for the linear elliptic equation was found in \cite  {Kato96}, 
and it was extended to nonlinear elliptic equation in \cite{Bla}. 
Here we adopt the proof from \cite{Kato96, Bla}.
By \cite{F58, Mo58}, $\zeta(x)$ is analytic when $x_n>0$. 
Here we show that $\zeta(x)$ is analytic in $x'$ when $x_n=0$. 
 
In \cite{Kato96}, Kato demonstrated his idea by considering the equation
\beq\label{Du2}
\Delta u =  u^2\ \ \text{in}\ \Omega,
\eeq
where  $\Omega$ is a domain in $\R^n$.
Instead of \eqref{grow1}, Kato's strategy is to establish the estimate
\beq\label{grow2}
\|\rho^{|\alpha|} \p_x^\alpha u\|_{H^m(B_{r_0})} \le C A^{|\alpha|} |\alpha| ! ,
\eeq
where $\rho$ is a cut-off function such that $\rho=1$ in $B_{r_0/2}$ and $\rho=0$ outside $B_{r_0}$.
He chooses $m=\big[\frac n2\big] +1$ such that $\|u\|_{L^\infty(B_{r_0})}\le C_{r_0}\|u\|_{H^m(B_{r_0})}$.
Hence \eqref{grow2} implies  \eqref{grow1}.
In \cite{Bla}, Blatt extended the estimate \eqref{grow2} to the general fully nonlinear, uniformly elliptic equation
\beq\label{fne}
\Phi (x, u(x), Du(x), D^2 u(x))=0\ \ \text{in}\ \Omega.
\eeq

The proof in \cite{Kato96} is rather simple,
one can easily see that  
the norm $H^m(B_{r_0})$ in \cite{Kato96, Bla} can be replaced by the H\"older space $C^\delta(B_{r_0})$, 
using the $C^{2,\delta}$ estimate (Schauder estimate) for  the Laplace equation.

To apply the argument in \cite{Kato96, Bla} to our equation \eqref{eqana1}, we use the H\"older norm 
$\|\cdot\|_{C^\delta(B_{r_0})}$ instead of the norm $\|\cdot\|_{H^m(B_{r_0})}$, and use the $C^{2,\delta}$ estimate 
(Theorem \ref{T5.4}).
As our equation contains the singular term $\frac{\zeta_n}{x_n}$, 
we cannot obtain the analyticity of $\zeta$ on $x_n$ (near $x_n=0$) by their simple  proof, 
but we can obtain the analyticity of $\zeta$ on $x'$, namely  
\beq\label{grow3}
\|\rho^{|\alpha|} \p_{x'}^\alpha \zeta \|_{C^\delta(B_{r_0})} \le C A^{|\alpha|} |\alpha| !
\eeq
for all multi-index $\alpha=(\alpha_1, \cdots, \alpha_{n-1})$.

\begin{proof}[Proof of Theorem \ref{thm6}]
As the proof is similar to that in \cite{Kato96, Bla}, we sketch the main idea only.
Let 
\beq\label{LiE}
\mathcal L[\phi]=:\phi_{nn}+b\frac{\phi_n}{x_n} 
+{\small\text{$\sum_{i+j<2n}$}} a_{ij}(x) \phi_{ij}+ {\small\text{$\sum_{i=1}^{n}$}} b_i(x) \phi_i+c(x) \phi
\eeq
be the linearized operator of \eqref{eqana1}, 
where $a_{ij}, b_i, c$ are functions of $x, \zeta, \p_i\zeta\ (1\le i\le n)$ and $\p^2_{ij}\zeta\ (i+j<2n)$,
 and $b=\frac{n+2}{n}$.
 From the proof of Lemma \ref{lemconti1}, we have $a_{in}(0)=0$ for $i \le n-1$. 
 Denote
$$
\mathcal L_0[\phi]=:\phi_{nn}+b\frac{\phi_n}{x_n}  +{\small\text{$\sum_{i,j<n}$}} a_{ij}(0) \phi_{ij}.
$$
 
Let  $\rho =\rho(|x'|)$ be a cut-off function of $x'$, such that
$\rho(x')=1$ when $|x'|< r_0/2$ and $\rho(x)=0$ when $|x'|>r_0$.
For any multi-indices $\alpha\in\R^{n-1}$ and $\beta\in\R^n$ with $|\alpha|=N-1\ge 1$ and $|\beta|=2$, 
as in \cite{Kato96, Bla}, we compute
\beq\label{6.7}
{\begin{split}
 & \| \rho^{N-1}  \p_x^\beta\p_{x'}^\alpha  u\|  
  \le  \| \p^\beta_{x} \big[\rho^{N-1} \p_{x'}^{\alpha} u\big]\| + \| \big[\p_{x}^\beta, \rho^{N-1}\big] \p_{x'}^{\alpha} u\| \\
  \le & C\| \mathcal L_0\big[ \rho^{N-1} \p_{x'}^{\alpha} u\big]\| + \| \big[\p_{x}^\beta, \rho^{N-1}\big] \p_{x'}^{\alpha} u\|+CA^{N+1}(N+1)! \\
  =&  C\| \mathcal L_0\big[ \rho^{N-1} \p_{x'}^{\alpha} u\big]\| + \| \big[\p_{x'}^\beta, \rho^{N-1}\big] \p_{x'}^{\alpha} u\|+CA^{N+1}(N+1)! \\
 \le   &C\| \rho^{N-1} \p_{x'}^{\alpha} \mathcal L_0[u]\| + C \| \big[\mathcal L_0, \rho^{N-1}\big] \p_{x'}^\alpha u\|
             + \| [\p_{x'}^\beta, \rho^{N-1}] \p_{x'}^{\alpha} u\| +CA^{N+1}(N+1)!\\
 =   &I_1 + I_2 + I_3+CA^{N+1}(N+1)!.
\end{split}}
\eeq
Here $ \|\cdot\| = \|\cdot\|_{C^\delta(\overline{B}'_{r_0} \times [0, r_0])}$, $B'_{r_0}=\{x'\in\R^{n-1}:\ |x'|<r_0\}$, and
$$\big[\p_{x}^\beta, \rho^{N-1}\big] \p_{x'}^{\alpha} u = 
\rho^{N-1}  \p_{x'}^\alpha \p_x^\beta u -  
  \p^\beta_{x} \big[\rho^{N-1} \p_{x'}^{\alpha} u\big]. $$
In the second inequality of \eqref{6.7}, we use the $C^{2,\delta}$ estimate (Theorem \ref{T5.4}) in the form
 \begin{equation}\label{609}
\|u\|_{C^{2,\delta}(\overline{B}'_{r_0} \times [0, r_0])}
             \le C\left(\|f\|_{C^{\alpha}(\overline{B}'_{2r_0} \times [0, 2r_0])}+\|u\|_{C^0(\overline{B}'_{2r_0} \times [0, 2r_0])}\right).
\end{equation}
But the cut-off function $\rho$ is supported in $B'_{r_0}$ and $u$ is analytic in the interior. 
Hence when applying \eqref{609} to \eqref{6.7}, 
the domain $\overline{B}'_{2r_0} \times [0, 2r_0]$ on the RHS of \eqref{609} can be replaced by $\overline{B}'_{r_0}\times [0,r_0]$ plus an additional term $CA^{N+1}(N+1)!$.

 The estimation for $I_1, I_2, I_3$ is the same as  in \cite{Kato96, Bla} and is omitted here. 
Note that by \eqref{6.7} and iteration, we not only obtain \eqref{grow3}, but also 
\beq\label{grow4}
\|\rho^{|\alpha|} \p_x^2 \p_{x'}^{\alpha-2} \zeta \|_{C^\delta(B_{r_0})} \le C A^{|\alpha|} |\alpha| ! .
\eeq
\eqref{grow4} is needed in the estimation of $I_1, I_2, I_3$ when differentiating equation \eqref{LiE}.
 \end{proof}
 
%%%% {\small \noindent{\it Acknowledegment:}
%%%% The first author was supported by NNSFC 11871160, 
%%%% the second author was supported by NNSFC 11831009 and NNSFC 12171185,
%%%% the third author was supported by ARC DP200101084.
%%%% The first and second authors started to work on this project when they were visiting 
%%%% the Mathematical Sciences Institute, Australian National University. 
%%%% They would like to thank the hospitality at ANU during their visit. }

%% \begin{theorem}\label{thm5}
%% Let $\phi$ be the tangent cone of $u$ at $0$.
%% Then  the section $S_{1,\phi}=\{x\in \R^n:\ \phi(x)<1\}$ is uniformly convex and analytic 
%% provided $g$ is positive and analytic.
%% \end{theorem}

 \end{document}